\documentclass[12pt, twoside]{article}
\usepackage{amsmath,amsthm,amssymb, amscd}
\usepackage[all]{xy}
\usepackage{graphicx}
\usepackage{mathrsfs}
\usepackage{amsmath}
\usepackage{pb-diagram}
\usepackage{times}
\usepackage{enumerate}

\def\titlerunning#1{\gdef\titrun{#1}}
\makeatletter
\def\author#1{\gdef\autrun{\def\and{\unskip, }#1}\gdef\@author{#1}}
\def\address#1{{\def\and{\\\hspace*{18pt}}\renewcommand{\thefootnote}{}%
\footnote {#1}}%
\markboth{\autrun}{\titrun}}
\makeatother
\def\email#1{e-mail: #1}
\def\subjclass#1{{\renewcommand{\thefootnote}{}%
\footnote{\emph{Mathematics Subject Classification (2010):} #1}}}
\def\keywords#1{\par\medskip
\noindent\textbf{Keywords.} #1}

\newtheorem{thm}{Theorem}[section]
\newtheorem{cor}[thm]{Corollary}
\newtheorem{lem}[thm]{Lemma}

\newtheorem{proposition}[thm]{Proposition}

\theoremstyle{definition}
\newtheorem{defin}[thm]{Definition}
\newtheorem{rem}[thm]{Remark}

\newtheorem{example}[thm]{Example}

\numberwithin{equation}{section}

\newcommand{\Hcal}{{{\mathcal{H}}}}

\newcommand{\integ}{\ensuremath{{\hbox{\tiny{int}}}}}

\newcommand{\Tt}{\operatorname{{\mathbb{T}}}}

\newcommand{\tU}{{\widetilde{\U}}}

\newcommand{\Ind}{Ind}

\newcommand{\Ker}{\operatorname{Ker}}
\newcommand{\Coker}{\operatorname{Coker}}

\newcommand{\Hom}{\operatorname{Hom}}
\newcommand{\End}{\operatorname{End}}

\newcommand{\Ann}{\operatorname{Ann}}
\newcommand{\Aut}{\operatorname{Aut}}

\newcommand{\MOD}{\ensuremath{{\operatorname{Mod}}}}
\newcommand{\mmOD}{\ensuremath{{\operatorname{mod}}}}

\newcommand{\Z}{\ensuremath{\operatorname{Z}}}

\newcommand{\isoto}{\ensuremath{\overset{\sim}{\longrightarrow}}}

\newcommand{\C}{\ensuremath{{\mathbb{C}}}}
\newcommand{\D}{\ensuremath{\mathcal{D}}}

\newcommand{\B}{\ensuremath{\mathcal{B}}}
\newcommand{\Q}{\ensuremath{\mathcal{Q}}}
\newcommand{\Po}{\ensuremath{\mathcal{P}}}
\newcommand{\N}{\ensuremath{\mathcal{N}}}

\newcommand{\g}{\ensuremath{\mathfrak{g}}}
\newcommand{\bo}{\ensuremath{\mathfrak{b}}}
\newcommand{\n}{\ensuremath{\mathfrak{n}}}
\newcommand{\uo}{\ensuremath{\mathfrak{r}}}
\newcommand{\lo}{\ensuremath{\mathfrak{l}}}
\newcommand{\ko}{\ensuremath{\mathfrak{k}}}

\newcommand{\p}{\ensuremath{\mathfrak{p}}}
\newcommand{\qo}{\ensuremath{\mathfrak{q}}}
\newcommand{\h}{\ensuremath{\mathfrak{h}}}

\newcommand{\w}{\ensuremath{\widehat}}
\newcommand{\wt}{\ensuremath{\widetilde}}

\newcommand{\Loc}{\ensuremath{\mathcal{L}}}
\newcommand{\BGG}{\ensuremath{\mathcal{O}}}
\newcommand{\BGGcat}{\ensuremath{\operatorname{O}}}

\newcommand{\pPB}{\ensuremath{\pi^\mathcal{P}_\mathcal{B}}}
\newcommand{\pQP}{\ensuremath{\pi^\mathcal{Q}_\mathcal{P}}}
\newcommand{\pQPls}{\ensuremath{\pi^\mathcal{Q}_{\mathcal{P}*}}}
\newcommand{\pQPus}{\ensuremath{\pi^{\mathcal{Q}*}_{\mathcal{P}}}}

\newcommand{\pQPt}{\ensuremath{\widetilde{\pi}^\mathcal{Q}_\mathcal{P}}}

\newcommand{\pQPtls}{\ensuremath{\widetilde{\pi}^\mathcal{Q}_{\mathcal{P}*}}}
\newcommand{\fin}{\ensuremath{{\hbox{\tiny{fin}}}}}
\newcommand{\Zfin}{\ensuremath{{\operatorname{Z}\hbox{\text{-}{\tiny{fin}}}}}}
\newcommand{\Zlfin}{\ensuremath{{\operatorname{Z}(\lo)\hbox{\text{-}{\tiny{fin}}}}}}

\newcommand{\W}{\ensuremath{\mathcal{W}}}

\newcommand{\Ugl}{\ensuremath{\operatorname{U}^\lambda}}
\newcommand{\U}{\ensuremath{\operatorname{U}}}
\newcommand{\Ul}{\ensuremath{\operatorname{U}^\lambda}}

\newcommand{\Ug}{\ensuremath{\operatorname{U}}}

\newcommand{\DtP}{\ensuremath{\widetilde{\mathcal{D}}_{\mathcal{P}}}}

\newcommand{\DtQ}{\ensuremath{\widetilde{\mathcal{D}}_{\mathcal{Q}}}}
\newcommand{\DtB}{\ensuremath{\widetilde{\mathcal{D}}_{\mathcal{B}}}}

\newcommand{\DtXL}{\ensuremath{\widetilde{\mathcal{D}}_{X/K}}}

\newcommand{\DPl}{\ensuremath{{\mathcal{D}}^{\lambda}_{\mathcal{P}}}}

\newcommand{\DBl}{\ensuremath{{\mathcal{D}}^{\lambda}_{\mathcal{B}}}}

\newcommand{\ot}{\ensuremath{\otimes}}

\newcommand{\gP}{\widetilde{\g}_\mathcal{P}}
\newcommand{\gPl}{\widetilde{\g}_\mathcal{P}^\lambda}
\newcommand{\gB}{\widetilde{\g}_\mathcal{B}}

\newcommand{\gQ}{\widetilde{\g}_\mathcal{Q}}

\newcommand{\Qo}{\mathcal{Q}}

\begin{document}


\baselineskip=17pt


\titlerunning{Singular Localization of $\mathfrak{g}$-modules}

\title{Singular Localization of $\mathfrak{g}$-modules and applications to representation theory}

\author{Erik Backelin
and Kobi Kremnizer \\ {\bf Accepted for publication in Journal of
the EMS}}

\date{}

\maketitle

\address{Erik Backelin: Departamento de Matem\'{a}ticas, Universidad de los Andes,
Carrera 1 N. 18A - 10, Bogot\'a, Colombia; \email{erbackel@uniandes.edu.co}
\and
Kobi Kremnizer: Mathematical Institute, University of Oxford, 24–29 St Giles'
Oxford OX1 3LB, UK; \email{kremnitzer@maths.ox.ac.uk}}

\subjclass{Primary 17B10, 14-XX}


\begin{abstract}
We prove a singular version of
Beilinson-Bernstein localization for a complex semi-simple Lie
algebra following ideas from the positive characteristic case
done by  \cite{BMR2}. We apply this theory to  translation functors, singular blocks in the Bernstein-Gelfand-Gelfand category $\BGGcat$ and Whittaker modules.

\keywords{Lie algebra, Beilinson-Bernstein localization, category $\BGGcat$}
\end{abstract}

\section{Introduction}
\subsection{} Let $\g$ be a semi-simple complex Lie algebra with enveloping algebra $\U$ and center $\Z \subset \U$.
Let $\h \subset \g$ be a
Cartan subalgebra and $\B$ be the flag manifold of $\g$. Let
$\lambda \in \h^*$ be regular and dominant and $I_\lambda \subset \Z$
be the corresponding maximal ideal determined by the Harish Chandra homomorphism.  Put $\Ugl :=
\U/ (I_\lambda)$.
Let $\DBl$ be the sheaf of
$\lambda$-twisted differential operators on $\B$. The celebrated
localization theorem of Beilinson and Bernstein, \cite{BB81},
states that the global section functor gives an equivalence
$\MOD(\DBl) \cong \MOD(\Ugl)$. For applications and more
information, see \cite{HTT08}.

A localization theory for singular $\lambda$ was much later found in positive characteristic by Bezrukavnikov,  Mirkovi\'c  and Rumynin, \cite{BMR2}.
Let us sketch their basic construction (which makes sense in all
characteristics):

Let $G$ be a semi-simple algebraic group such that $Lie \, G = \g$.
Instead of $\B$ consider a parabolic flag manifold $\Po = G/P$,
where $P \subseteq G$ is a parabolic subgroup whose parabolic
roots coincide with the singular roots of $\lambda$. Replace the
sheaf $\DBl$ by a sheaf $\DPl := \pi_*( \D_{G/R})^{L}$ modulo a
certain ideal defined by $\lambda$. Here $L$ is the Levi factor
and $R$ is the unipotent radical of $P$ and $\pi: G/R \to \Po$ is
the projection. The $L$-invariants are taken with respect to the
right $L$-action on $G/R$. The sheaf $\pi_*( \D_{G/R})^{L}$ is
locally isomorphic to $\D_\Po \otimes \U(\lo)$, where $\lo = Lie
\, L$. When $P=B$ we have $\DPl = \DBl$ and when $P=G$ we arrive
at a tautological solution: $\DPl =$ $\Ugl \ot$  ``sheaf of
differential operators on a point" $= \Ugl$.

\medskip

\noindent
We use this construction to prove a singular
localization theorem in characteristic zero, Theorem
\ref{singlocthmtag0}.
This is probably well known to the experts but it isn't
in the literature. Our proof is similar to the original proof
of \cite{BB81}, though parabolicity leads to some new complications. For instance,
\cite{BB81} introduced the method of tensoring a $\D_\B$-module with a trivial bundle and then to filter
this bundle with $G$-equivariant line bundles as  subquotients. In the parabolic setting the subquotients will necessarily be vector bundles - which are harder to control -
since irreducible representations of $P$ are generally not one-dimensional.

\smallskip

In Theorem \ref{thmtag0}  we show that global section $\Gamma(\DPl)$ equals $\Ul$ by passing to the associated graded level,
i.e. to the level of a
parabolic Springer resolution. That this works we deduce from the usual Springer resolution, Lemma \ref{classicalglobal}.

Our localization theorem gives an equivalence at the level of
abelian categories just like \cite{BB81} does.
This is different from positive characteristic where the
localization theorem only holds at the level of derived
categories.

\subsection{}
Our principal motivation comes from quantum groups.  We do not wish to get into details here, but let us at least mention that we will need a singular localization theory for quantum groups
in order to establish quantum analogs of fundamental constructions  from  \cite{BMR, BMR2, BM10} that relate modular representation theory to (commutative) algebraic geometry.
By our previous work, \cite{BK08},  we know that the derived representation categories of quantum groups at roots of unity
are equivalent to derived categories of coherent sheaves on Springer fibers in $T^*\B$.

To extend this to the level of abelian categories we must transport the tautological $t$-structure
on the representation theoretical derived category to a $t$-structure on the coherent sheaf side. It so happens that to describe this so called \textit{exotic} $t$-structure (see also \cite{Bez06}) a family of singular
localizations is needed (even for a regular block).

\smallskip

We showed in \cite{BK06} that
a localization theory for quantum groups can be neatly formulated in terms of equivariant sheaves.
The
``space" $G/B$ doesn't admit a quantization.  However, one can quantize function algebras $\BGG(G)$ and
$\BGG(B)$ and thus the category of $B$-equivariant (= $\BGG(B)$-coequivariant)
$\BGG(G)$-modules. This is just the category of quasicoherent sheaves on $G/B$.  Therefore, to prepare for the quantum case we have taken thorough care to write down our results in an equivariant categorical language and at the same time
to explain what is going on geometrically while this is still possible.

\subsection{}
The theory of singular localization of $\g$-modules clarifies many aspects of representation theory and will have many applications in its own right.  Here we discuss a few of them.

\smallskip

It is a basic principle in representation theory that understanding of representations at singular central characters enhances the understanding also at regular central characters.
This is illustrated by our $\D$-module interpretation of translation functors (Section \ref{Translation functors}).
Using regular localization only, such a theory was developed by Beilinson and Ginzburg, \cite{BeGin99}.
Singular localization simplifies their picture for the plain reason that wall-crossing functors between regular blocks factors through a singular block.
We shall also need these results in our work on quantum groups.

\smallskip

The localization theorem implies that a (perhaps singular) block $\BGGcat_\lambda$ in
category $\BGGcat$ corresponds to
certain bi-equivariant $\D$-modules on $G$ (Section
\ref{section O}). From this we directly retrieve Bernstein and Gelfand's, \cite{BerGel} , classic result that $\BGGcat_\lambda$ is equivalent to a category of Harish-Chandra bimodules, Corollary \ref{singparcor}.

Singular localization also leads to the useful observation that one should study Harish-Chandra $\g\text{-}\lo$-bimodules,
where $\lo$ is the Levi factor of $\p = Lie \, P$, rather than $\g\text{-}\g$-bimodules (as well as the only proof we know that such bimodules are equivalent to $\BGGcat_\lambda$.)
For instance, Theorem \ref{MainWhittakerTheorem} gives this way a very short proof for  Mili\v ci\'c and Soergel's equivalence between
$\BGGcat_\lambda$ and a block in the category of Whittaker modules, \cite{MS97}, and Corollary \ref{parabolic Whittaker cor} gives one for its parabolic generalization due to Webster, \cite{W09}.
These Whittaker categories have encountered recent interest because they are equivalent to modules over finite $W$-algebras, e.g. \cite{W09}.  It is probably well worth the effort to further investigate the
relationship between singular localization and finite $W$-algebras; in particular so in the affine case.

We also retrieve and generalize some other known equivalences between representation categories,  e.g. \cite{Soe86}.

\subsection{}
An interesting task will be to develop a theory for ``holonomic" $\DPl$-modules. Those which are ``smooth along the Bruhat stratification of $\Po$" and have ``regular singularities"
will correspond to $\BGGcat_\lambda$. One should then establish a ``Riemann-Hilbert correspondence" between holonomic $\DPl$-modules with regular singularities and a suitable category
of constructible sheaves on $\Po$. Ideally the latter category would be accessible to the machinery of Hodge theory. This would further strengthen the interplay between representation theory and algebraic topology.
Because of the simple local description of $\DPl$ we believe that all this can be done and is a good starting point for generalizing $\D$-module theory. We shall return to this topic later on.

\smallskip

Another topic we would like to approach via singular localization is the singular-parabolic Koszul duality for $\BGGcat$ of \cite{BGS}.




\section{Preliminaries}
Here we fix notations and collect mostly well known results that
we shall need.
\subsection{Notations}\label{notations} We work over $\C$. Unless stated otherwise, $\ot = \ot_\C$. Let $X$ be an algebraic
variety, $\BGG_X$ the sheaf of regular functions on $X$ and
$\BGG(X)$ its global sections. $\MOD(\BGG_X)$ denotes the category
of quasi-coherent sheaves on $X$ and $\Gamma := \Gamma_X:
\MOD(\BGG_X) \to \MOD(\BGG(X))$ is the global section functor.
If $Y$ is another variety $\pi^Y_X$ will denote the obvious projection
$X \to Y$ if there is a such.

\smallskip

For $\mathcal{A}$ a sheaf of algebras on $X$  such that $\BGG_X
\subseteq \mathcal{A} $ (e.g., an algebra if $X = pt$) we
abbreviate an $\mathcal{A}$-module for a sheaf of
$\mathcal{A}$-modules that is quasi-coherent over $\BGG_X$. We
denote by $\MOD(\mathcal{A})$ the category of $\mathcal{A}$-modules.
More generally, we will encounter categories such as
$\MOD(\mathcal{A},\, additional \, data)$ that consists of
$\mathcal{A}$-modules with some $additional \, data$. We will then
denote by $\mmOD(\mathcal{A},\, additional \, data)$ its full
subcategory of noetherian objects.

\smallskip

Throughout this paper $G$ will denote a semi-simple complex linear algebraic group. We have assumed semi-simplictly to simplify notations;
all our results can be straightforwardly extended to the case that $G$ is reductive. We remark on this fact in those proofs that reduce to (reductive) Levi subgroups of $G$.

\subsection{Root data}
Let $B \subset
G$ be a Borel subgroup of our semi-simple group $G$ and let $T \subset B$ be a maximal torus. Let $\h
\subset \bo \subset \g $ be their respective Lie algebras. For any
parabolic subgroup $P$ of $G$ containing $B$, denote by $R =
R_P$ its unipotent radical and by $L := L_P$ its Levi subgroup and
by $\p = Lie \, P$, $\uo = \uo_P = Lie \, R$ and $\lo = \lo_P = Lie \, L$ their Lie algebras. We denote
by $\B := G/B$ the flag manifold and by $\Po := G/P$ the parabolic
flag manifold corresponding to $P$.

Let $\Lambda$ be the lattice of integral weights and let
$\Lambda_r$ be the root lattice. Let $\Lambda_+$ and $\Lambda_{r
+}$ be the positive weights and the positive linear combinations
of the simple roots, respectively.

Let  $\W$ be the Weyl group of $\g$. Let $\Delta$ be the simple
roots and let $\Delta_P: = \{\alpha \in \Delta: \g^{-\alpha}
\subset \p\}$ be the subset of $P$-parabolic roots. Let $\W_P$ be
the subgroup of $\W$ generated by simple reflections $s_\alpha$,
for $\alpha \in \Delta_P$. Note that $\h$ is a Cartan subalgebra
of the reductive Lie algebra $\lo_P$. Denote by $S(\h)^{\W_P}$ the
$\W_P$-invariants in $S(\h)$ with respect to the $\bullet$-action
(here $w \bullet \lambda := w(\lambda + \rho) - \rho$, for
$\lambda \in \h^*$, $w \in \W$, $\rho$ is the half sum of the
positive roots ).

Let $\Z(\lo)$ be the center of $\U(\lo)$ and put $\Z := \Z(\g)$. We have the Harish-Chandra homomorphism
$S(\h)^{\W_P} \cong \Z(\lo)$ (thus
$S(\h)^\W \cong \Z$).

Put $\Delta_\lambda := \{\alpha \in \Delta; \lambda(H_\alpha) =
-1\}$, $\lambda \in \h^*$, where $H_\alpha \in \h$ is the coroot
corresponding to $\alpha$. Let $\chi_{\lo,\lambda}: \Z(\lo)
\to \C$ be the character such that $I_{\lo,\lambda} := \Ker \chi_{\lo,\lambda}$
annihilates the Verma module $M_{\lambda}$ (for $\U(\lo)$) with highest weight
$\lambda$. Thus, $\chi_{\lo,\lambda} = \chi_{\lo,\mu} \iff \mu
\in \W_P \bullet \lambda$. We have $\lambda = \chi_{\h,\lambda}$
and we write $\chi_\lambda := \chi_{\g,\lambda}$ and $I_\lambda := \Ker \chi_\lambda$.

Let $\U := \U(\g)$ be the enveloping algebra of $\g$ and $\tU
:= \U \ot_{\Z} S(\h)$ the extended enveloping algebra; thus
$\tU$ has a natural $\W$-action such that the invariant ring
${\tU}{}^\W$ is canonically isomorphic to $\U$. Let $\Ul :=
\U/(I_\lambda)$. We say that
     \begin{itemize}
\item $\lambda \in \h^*$ is {\it $P$-dominant} if
$\lambda(H_\alpha) \notin \{-2,-3,-4, \ldots\}$, for $\alpha \in
\Delta_P$; $\lambda$ is dominant if it is $G$-dominant. \item
$\lambda$ is $P$-{\it regular} if $\Delta_\lambda \subseteq
\Delta_P$. $\lambda$ is regular if it is $B$-regular, that is if
$w \bullet \lambda = \lambda \implies w = e$, for $w \in \W$.
\item $\lambda$ is a $P$-\emph{character} if
it extends to a character of $P$; thus $\lambda$ is a
$P$-character iff $\lambda$ is integral and $\lambda |_{\Delta_P}
= 0$.

 \end{itemize}
Suppose now that $\lambda \in \h^*$ is integral and $P$-dominant.
Then there is an irreducible finite dimensional $P$-representation
$V_P(\lambda)$ with highest weight $\lambda$. Note that
$V_{L}(\lambda) := V_P(\lambda)$ is an irreducible representation
for $L$. Of course, $\dim V_P(\lambda) = 1 \iff \lambda$ is a
$P$-character.

The following is well-known:   \begin{lem}\label{regdomlem} Let
$\lambda \in \h^*$. Then $\lambda$ is dominant iff for all $\mu
\in \Lambda_{r +} \setminus \{0\}$ we have $\chi_{\lambda + \mu}
\neq \chi_\lambda$.
\end{lem}
We also have
\begin{lem}\label{Pregularweightlemma}  Let $\lambda \in \h^*$ be $P$-regular and dominant.
Let $\mu$ be a $P$-character and let $V$ be the finite dimensional
irreducible representation
         of $\g$ with extremal weight  $\mu$. Then for any weight $\psi$ of $V$, $\psi \neq \mu$,
         we have $\chi_{\lambda + \mu} \neq \chi_{\lambda +\psi}$.
         \end{lem}
\begin{proof} This is well known for $P = B$. We reduce to that case as
follows: Let $\g'$ be the semi-simple Lie subalgebra of $\g$
generated by $X_{\pm \alpha}$, $\alpha \in \Delta \setminus
\Delta_P$. Let $\h' := \g' \cap \h$ be the Cartan subalgebra of
$\g'$. The inclusion $\h' \hookrightarrow \h$ gives the projection
$p: \h^* \to {\h'}^*$. Consider the restriction $V \vert_{\g'}$ of
$V$ to $\g'$ and let $V'$ denote the irreducible $\g'$-module with
highest weight $p(\mu)$; $V'$ is a direct summand in $V
\vert_{\g'}$. Let $\Lambda(V)$ denote the set of weights of $V$.
Then $p(\Lambda(V)) = \Lambda'(V \vert_{\g'})$, the weights of $V
\vert_{\g'}$. By the assumption that $\mu$ is a $P$-character, it
follows that $p(\Lambda(V))$ is contained in the convex hull
$\overline{\Lambda'(V')}$ of $\Lambda'(V')$. Since $p(\lambda)$ is
regular and dominant it is well known that $p(\lambda + \mu)
\notin \W'(p(\lambda) +  \Lambda(V'))$. But
then it follows that $p(\lambda + \mu) \notin \W'(p(\lambda) +
 \overline{\Lambda(V')})$. Now $\W' = p(\W)$,
so it follows that $\lambda + \mu \notin \W(\lambda +  \Lambda(V))$.
\end{proof}

\subsection{Equivariant $\BGG$-modules and
induction}\label{section induction} See \cite{Jan2} for details on
this material.

Let $K$ be a linear algebraic group and $J$ a closed algebraic
subgroup. For $X$ an algebraic variety equipped with a right (or
left) action of $K$ we denote by $\MOD(\BGG_{X},K)$ the category
of $K$-equivariant sheaves of (quasi-coherent) $\BGG_X$-modules. For $M \in \MOD(\BGG_{X},K)$ there is
the sheaf $(\pi^{X/K}_{X*}M)^K$ on $X/K$ of $K$-invariant local sections in
the direct image $\pi^{X/K}_{X*}M$.
If the $K$-action is free and the quotient is nice we have the
equivalence
$$[\pi^{X/K}_{X*}( \ )]^K: \MOD(\BGG_{X},K) \to \MOD(\BGG_{X/K}): \pi^{X/K *}_X.$$

%
We denote by $\Gamma_{(K,J)}$ the global section functor on
$\MOD(\BGG_K,J)$ that corresponds to
$\Gamma_{K/J}$ under the equivalence
$\MOD(\BGG_K,J) \cong \MOD(\BGG_{K/J})$. Then $\Gamma_{(K,J)}(M) =
M^J$, for $M \in \MOD(\BGG_K,J)$. 

\smallskip

\smallskip

\noindent Let $Rep(K)$ denote the category of algebraic
representations of $K$. We have $\BGG(K) \in Rep(K)$, via $(gf)(x)
:= f(g^{-1}x)$, for $g,x \in K$ and $f \in \BGG(K)$. We shall also
consider the left $J$-action on $\BGG(K)$ given by $(kf)(x) :=
f(xk)$, for $k \in J, x \in K$ and $f \in \BGG(K)$. These actions
commute.

For $V \in Rep( J)$ we consider the diagonal left $J$-action on
$\widetilde{V} := \BGG(K) \otimes V$. The left $K$-action on
$\BGG(K)$ defines a left $K$-action on $\widetilde{V}$ that
commutes with the $J$-action and the multiplication map $\BGG(K)
\otimes \widetilde{V} \to \widetilde{V}$ is $K$- and $J$-linear.
Thus $\widetilde{V}$ belongs to the category $\MOD(K, \BGG(K), J)$
of $K\text{-}J$ bi-equivaraint $\BGG(K)$-modules. This gives the
functor
$$p^*: Rep(J) \to \MOD(K, \BGG(K), J),\; V \mapsto \widetilde{V}$$
(induced bundle of a representation, $p$ symbolizes projection
from $K$ to $pt/J$).

Let $Ind^K_{J} V := \widetilde{V}^J \in Rep(K)$.

We have the factorization $Ind^K_{J} = ( \ )^J \circ p^*$.
 One can show that $R( \ )^J \circ p^*
\cong RInd^K_{J}$ where $R( \ )^J$ and $RInd^K_{J}$ are computed
in suitable derived categories. An important formula is the
\emph{tensor identity}
\begin{equation}\label{tensor identity}
RInd^{K}_{J} (V \otimes W) \cong RInd^{K}_{J} (V) \otimes W, \;
for \; V \in Rep(J), W \in Rep(K).
\end{equation}
(In particular $RInd^{K}_{J} (W) \cong W \ot RInd^K_J(\C)$, for $W
\in Rep(K)$ and $\C$ the trivial representations.)

\section{Parabolic Springer Resolutions}
In order to treat sheaves of extended differential operators on
parabolic flag varieties in the next section we will here gather
information about their associated graded objects. This is encoded
in the geometry of the parabolic Grothendieck-Springer resolution.
\subsection{Parabolic Flag Varieties}

The parabolic flag variety  $\Po$ has a natural left $G$-action. There is a bijection
between representations of $P$ and $G$-equivariant vector bundles
on $\Po$; a representation $V$ of $P$ correspond to the induced
bundle $G \times_P V$ on $\Po$. We denote by $\BGG(V) :=
\BGG_\Po(V)$ the corresponding locally free sheaf on $\Po$ which
hence has a left $G$-equivariant structure.

Let $\lambda \in \h^*$ be a $P$-character and write $\BGG(\lambda)
:= \BGG(V_P(\lambda))$ for the line-bundle corresponding to the
one-dimensional $P$-representation $V_P(\lambda)$. We have
$Pic(\Po)=Pic_G(\Po) \cong$ group of $P$-characters, (but note
that not all vector bundles on $\Po$ are $G$-equivariant). The
ample line bundles $\BGG(-\mu)$ are given by $P$-characters $\mu$
such that $\mu(H_\alpha) > 0$ for all $\alpha \in \Delta \setminus
\Delta_P$.

Next we define the parabolic Grothendieck resolution:
\begin{defin}\label{Grothresolution}
$\gP:=\{(P',x):P'\in \Po, x\in \g^*, x\vert_{\uo_{P'}}=0\}$
\end{defin}

Note that $\gP=G\times_P (\g/\uo_P)^*$. Recall that $L = L_P$ is
the Levi factor of $P$, $U = U_P$ its unipotent radical and $\lo =
\lo_P, \uo = \uo_P$ their Lie algebras. We have a commutative
square:

\begin{equation}\label{P diagram}
\begin{diagram}
\node{\gP} \arrow{e}  \arrow{s} \node{\lo^*/L=\h^*/\W_P} \arrow{s} \\
 \node{ \g^*} \arrow{e} \node{\h^*/\W}
\end{diagram}
\end{equation}
where the top map sends $(P',x)$ to
$x\vert\lo_{P'}/L_{P'}\in\lo_{P'}^*/L_{P'}\cong\lo^*/L$. Note that
the isomorphism $\lo_{P'}^*/L_{P'}\cong\lo^*/L$ is
canonical. (We can call $\lo^*/L$ the universal coadjoint
quotient of the Levi Lie subalgebra.)

This induces a map:
\begin{equation}
\pi_\Po:\gP\to \g^*\times_{\h^*/\W} \h^*/\W_P.
\end{equation}

\begin{lem}\label{classicalglobal}
$R\pi_{\Po*} \BGG_{\gP}=\BGG_{\g^*\times_{\h^*/\W} \h^*/\W_P}$.
\end{lem}
\begin{proof} We shall reduce to the well known case of the
ordinary Grothendieck resolution for $\Po = \B$. It states that
\begin{equation}\label{Grothendieck sheaves B}
R\pi_{\B*} \BGG_{\gB} = \BGG_{\g^*\times_{\h^*/\W} \h^*}.
\end{equation}
Translating this to the equivariant language it reads:
\begin{equation}\label{Grothendieck equivariant B}
RInd^G_B (S(\g/\n)) = S(\g) \otimes_{S(\h)^\W} S(\h).
\end{equation}
where $\n := [\bo,\bo]$. To see this, observe first that, since
$\g^*\times_{\h^*/\W} \h^*$ is affine, the equality
\ref{Grothendieck sheaves B} is after taking global sections
equivalent to the equality $$R\Gamma(\BGG_{\gB}) =
\BGG(\g^*\times_{\h^*/\W} \h^*) = S(\g) \otimes_{S(\h)^\W} S(\h)$$
of $G$-modules. Moreover, since the bundle projection $p: \gB \to
\B$ with fiber $(\g/\n)^*$ is affine, $p_*$ is exact and hence
$R\Gamma(\BGG_{\gB}) = R\Gamma(p_{*}(\BGG_{\gB}))$. Under the
identification $\MOD(\BGG_\B) = \MOD(\BGG_G,B)$ we have that
$p_{*}(\BGG_{\gB})$ corresponds to $S(\g/\n) \otimes \BGG(G)$ so
its derived global sections are given by $RInd^G_B (S(\g/\n))$ as
stated. This proves \ref{Grothendieck equivariant B}.

By a similar argument the statement of the lemma is equivalent to
proving that
\begin{equation}\label{Grothendieck equivariant B234}
R Ind^G_P (S(\g/\uo)) = S(\g) \otimes_{S(\h)^\W} S(\h)^{\W_P}.
\end{equation}
For any $M \in \MOD(B)$ we have an equality of $P$-modules
\begin{equation}\label{general equality}
 R \Ind^P_B(M) =  R \Ind^{L}_{L \cap B}(M).
\end{equation}
where the $R$-module structure on the RHS is defined by $(xf)(g)
:= g^{-1}xg \cdot f(g)$ for $f \in Mor(L,M)^{L \cap B} \cong
\Ind^{L}_{L \cap B}(M)$, $x \in U$, $g \in L$. Together with the
given $L$-action this makes the RHS a $P$-module. In particular we
have
\begin{equation}\label{some third step}
 R \Ind^P_B(S(\g / \n)) =  R \Ind^{L}_{L \cap B}(S(\g / \n)).
\end{equation}
We have a decomposition $\g = \overline{\uo}_P \oplus \lo \oplus
\uo$, where $\overline{\uo}_P$ is the image of  $\uo$ under the
Chevalley involution of $\g$; thus $\g/\n = \lo/(\lo \cap \n)
\oplus \overline{\uo}_P$. Thus
\begin{equation}\label{some other step}
R \Ind^{L}_{L \cap B}(S(\g / \n)) = R \Ind^{L}_{L \cap B}(S(\lo/
\lo \cap \n) \otimes S(\overline{\uo}_P)) =
\end{equation}
$$R
\Ind^{L}_{L \cap B}(S(\lo/ \lo \cap \n)) \otimes
S(\overline{\uo}_P) = S(\g/\uo) \otimes_{S(\h)^{\W_P}} S(\h)$$
where the last equality is given by \ref{Grothendieck equivariant
B} applied to $G$ replaced by $L$ and the second equality is the
tensor identity which applies since $S(\overline{\uo}_P)$ is an
$L$-module. Since $R Ind^G_B = R Ind^G_P \circ R Ind^P_B$ we get
from \ref{Grothendieck equivariant B}, \ref{some third step} and
\ref{some other step} that
$$
 S(\g) \otimes_{S(\h)^\W} S(\h) = R Ind^G_P (S(\g/ \uo) \otimes_{S(\h)^{\W_P}} S(\h)) = R Ind^G_P (S(\g/ \uo)) \otimes_{S(\h)^{\W_P}}
 S(\h).
$$
Since $S(\h)$ is faithfully flat over $S(\h)^{\W_P}$ this implies
\ref{Grothendieck equivariant B234}.
\end{proof}

Let $P\subset Q$ be two parabolic subgroups. The projection $\pQP:
\Po\to\Q$ induces a map $\pQPt:\gP \to \gQ$ that fits into the
following commutative square:

\begin{equation}\label{P2 diagram}
\begin{diagram}
\node{\gP} \arrow{e}  \arrow{s,r}{\pQPt} \node{\lo^*/L=\h^*/\W_P} \arrow{s} \\
 \node{ \gQ} \arrow{e} \node{\lo^*_Q/L_Q=\h^*/\W_Q}
\end{diagram}
\end{equation}
With similar arguments as in the proof of Lemma
\ref{classicalglobal} one can prove
\begin{lem}\label{relativeclassicalglobal}
$R {\pQPtls}  \BGG_{\gP}=\BGG_{\gQ\times_{\h^*/\W_Q} \h^*/\W_P}$.
\end{lem}

We observe that $\gP$ is an $L$-torsor over $T^*\Po$. We put
\begin{defin}
$\gPl=\gP\times_{\h^*/\W_P} \lambda$, for $\lambda \in \h^*$.
\end{defin}

We would like to view $\gPl$ as the classical Hamiltonian of
$T^*(G/R)$ with respect to the (right) $L$-action. We have a
moment map $\mu:T^*(G/R)\to \lo^*$. Recall that we can take the
Hamiltonian reduction with respect to any subset of $\lo^*$ stable
under the coadjoint action. Let $\mathcal{N}_\lambda\subset\lo^*$
be the preimage of $\lambda/\W_P\in\h^*/\W_P\cong  \lo^*_P/L$
under the quotient map. Then
\begin{equation}
T^*(G/R)//_{\mathcal{N}_\lambda} L=
\mu^{-1}(\mathcal{N}_\lambda)/L=\gPl.
\end{equation}

Note that we could also reduce with respect to $\lambda\in
(\lo^*)^{L}$ in which case we would get twisted cotangent bundles.

\section{Extended differential operators on $\Po$}
In this section we construct the sheaf of extended differential
operators on a parabolic flag manifold and describe its global
sections.
\subsection{Torsors} Let $X$ be an algebraic variety equipped with
a free right action of a linear algebraic group $K$ and let $p: X
\to X/K$ be the projection. We assume that $X$, locally in the
Zariski topology, is of the form $Y \times K$, for some variety
$Y$, and $p$ is first projection. Such $X$ is called an
$K$-torsor. We get induced right $K$-actions on the sheaf $\D_X$
of regular differential operators on $X$ and on the direct image
sheaf $p_*(\D_X)$. Denote by $\DtXL := p_*(\D_X)^K$ the sheaf on
$X/K$ of $K$-invariant local sections of $p_*(\D_X)$.

Let $\ko := Lie \, K$. The infinitesimal $K$-action gives algebra
homomorphisms $\hat{\epsilon}: U(\ko) \to \D_X$ and
$\tilde{\epsilon}: U(\ko) \to p_*\D_X$, which are injective since
the $K$-action is free. It follows from the definition of
differentiating a group action that $[\tilde{\epsilon}(U(\ko)),
\DtXL] = 0$.

Notice that $\tilde{\epsilon}(\U(\ko)) \nsubseteq \DtXL$, unless
$K$ is abelian, but $\tilde{\epsilon}(\Z(\ko)) \subseteq \DtXL$.
We denote by $\epsilon: \Z(\ko) \to \DtXL$ the restriction of
$\tilde{\epsilon}$ to $\Z(\ko)$. By the discussion above it is a
central embedding.

Now, using that $p$ is locally trivial we can give a local
description of $\DtXL$. Let $Y \times K$ be a Zariski open subset
of $X$ over which $p$ is trivial. Then $\D_X \vert_{Y \times K} =
\D_Y \otimes \D_K$ and $\DtXL \vert_{Y} = \D_Y \otimes \U(\ko)$,
where $\U(\ko)$ is identified with the algebra of right
$K$-invariant differential operators $\D^K_K$ on $K$.

Note that $\tilde{\epsilon}(\U(\ko)) \vert_{Y \times K} = 1 \ot
{}^K\D_K$ is the algebra of left $K$-invariant differential
operators on $Y \times K$, with respect to the natural left
$K$-action on $Y \times K$, that are constant along $Y$. Since
$\Z({}^K\D_K) = \Z(\D^K_K)$ we get that $\epsilon$ is locally
given by the embedding
$$
\Z(\ko) \hookrightarrow \U(\ko) \cong 1 \otimes \U(\ko)
\hookrightarrow \D_Y \otimes \U(\ko).
$$
This implies that $\epsilon(\Z(\ko)) = \Z(\DtXL)$.

Denote by $\MOD(\D_X,K)$ the category of weakly equivariant
$(\D_X,K)$-modules. In order to simplify the description of this
category we assume henceforth that $X$ is quasi-affine. Its object
$M$ is then a left $\D_X$-module equipped with an algebraic right
action $\rho := \{\rho_U\}$, where $\rho_U: K \to
\Aut_{\C_U}(M(U))^\mathrm{op}$ are homomorphism compatible with
the restriction maps in $M$, for each Zariski-open $K$-invariant
subset $R$ of $X$. We require that $\D_X \otimes M \to M$ is
$K$-linear (over $K$-invariant open sets) with respect to the
diagonal $K$-action on a tensor. (For a general $X$, $\rho$ must
be replaced by a given isomorphism $pr^*M \cong act^*M$ satisfying
a cocycle condition, where $pr$ and $act: X \times K \to X$ are
projection and the action map, respectively.)

Denote by $\MOD(\D_X,K,\ko)$ the category of strongly
equivariant $(\D_K,K)$-modules. Its object $(M, \rho)$ is a weakly
equivariant $(\D_X,K)$-module such that $d\rho(x)m =
\hat{\epsilon}(x)m$ for $x \in \ko$ and $m \in M$.

For $M \in \MOD(\D_X,K)$ we consider the sheaf $(p_* M)^K$
of $K$-invariant local sections in $p_* M$; it has a natural
$\DtXL$-module structure. Thus we get a functor $p_*$ whose right
adjoint is $p^*$ (the pullback in the category of $\BGG$-modules
with its natural equivariant structure). The following is standard
(see \cite{BB93}):
\begin{lem}\label{general D equivariance iso} The functors $i) \ p_* ( \ )^K: \MOD(\D_X,K) \leftrightarrows
\MOD(\DtXL): p^*$ and
$
ii) \ p_* ( \ )^K: \MOD(\D_X,K,\ko) \leftrightarrows
\MOD(\D_{X/K}): p^*
$
are mutually inverse equivalences of categories.
\end{lem}

\subsection{Definition of extended differential operators}\label{Definition of extended differential operators}

On $G/R$ we shall always consider the right $L$-action
$(\overline{g},h) \mapsto \overline{gh}$, for $g \in G$ and $h \in
L$. Thus, $G/R$ is an $L$-torsor. We put
\begin{defin}\label{extended diff op}
$\DtP:=\pi^{\Po}_{G/R*}(\D_{G/R})^{L}$.
\end{defin}
By the results of the previous section we have that locally on
$\Po$, $\DtP \cong \D_{\Po} \otimes \U(\lo)$, and we have the
central algebra embedding $\epsilon: \Z(\lo) \to \DtP$.

For $\lambda \in \h^*$ we define:

\begin{defin}\label{extended diff op lambda}
$\DPl:=\DtP\otimes_{\epsilon(\Z(\lo))}\C_\lambda$.
\end{defin}

\subsection{Equivariant description.}\label{Equivariant description}   For any $\Z(\lo)$-algebra $S$ and $\lambda \in \h^*$ let $\MOD^{\w\lambda}(S)$
be the category of left $S$-modules which are locally annihilated by some power of $I_{\lo,\lambda}$.

We shall give equivariant descriptions on $G$ and on $G/R$ of the category
 $\MOD(\DtP)$ and its subcategories $\MOD(\DPl)$ and $\MOD^{\w{\lambda}}(\DtP)$. It is best to work on $G$. We start with $G/R$ as an intermediate step.

By Lemma \ref{general D equivariance iso} we have mutually inverse
equivalences
\begin{equation}\label{equivariant equvalence 4}
\pi^{\Po}_{G/R*}( \ )^{L}: \MOD(\D_{G/R}, L)
\leftrightarrows \MOD(\DtP): \pi^{\Po*}_{G/R}.
\end{equation}
Differentiating the right $L$-action on $G/R$ gives an algebra embedding  $\U(\lo) \hookrightarrow \D_{G/R}$. This allows us to consider $\Z(\lo) \subseteq \U(\lo)$ as a
subalgebra of $\D_{G/R}$.
Transporting conditions from the right-hand side to the left-hand
side of \ref{equivariant equvalence 4} we see that $\MOD(\DPl)$ is equivalent to the full subcategory
$\MOD(\D_{G/R},L, {{\lambda}})$  of $\MOD(\D_{G/R},L)$
whose object $M$ satisfy $I_{\lo, \lambda} \cdot M^{L} = 0$.
Similarly, $\MOD^{\w{\lambda}}(\DtP)$ is equivalent to the
full subcategory $\MOD(\D_{G/R},L, \w{{\lambda}})$
of $\MOD(\D_{G/R},L)$ whose object $M$ satisfies that $I_{\lo, \lambda}$ is locally nilpotent on  $M^{L}$.

\smallskip

\noindent Now we pass to $G$. Let us introduce some notations:

We have a left and right actions $\mu_l$ and $\mu_r$ of $G$ on
$\BGG(G)$ defined by $\mu_l(g)f(h) := f(g^{-1}h)$ and $\mu_r(g)f(h)
:= f(hg^{-1})$, for $f \in \BGG(G), g,h \in G$, respectively.
Differentiating $\mu_l$, resp., $\mu_r$, gives an injective
algebra homomorphism $\epsilon_l: \Ug \to \D_G$, resp., an
anti-homomorphism $\epsilon_r: \Ug \to \D_G$. We have that
$\epsilon_l(\Ug) = \D^G_G$ consists of \textit{right} invariant
differential operators on $G$ and
 $\epsilon_r(\Ug) = {}^G\D_G$ consists of \textit{left} invariant differential operators on
 $G$, $\Z = \epsilon_l(\Ug) \cap \epsilon_r(\Ug)$ and $\epsilon_l |_{\Z} = \epsilon_r |_{\Z}$.

The actions $\mu_l$ and $\mu_r$ induce left and right actions of
$G$ on $\D_G$ that we denote by the same symbols.

Let $\MOD(\D_G,P,\uo)$ be the category whose objects are ($M,\rho$)
where
\begin{enumerate}[(1)]
\item  $M$ is a left $\D_G$-module.
\end{enumerate}
 \begin{enumerate}[(2)]
 \item  $\rho$ is a right algebraic $P$-action on $M$ such that $\D_G \otimes M \to M$ is
$P$-linear, with respect to the right $P$-action $\mu_r |_P$ on
$\D_G$ and the diagonal $P$-action on the tensor product.
\end{enumerate}
 \begin{enumerate}[(3)]\item $d\rho
|_{\uo} = \epsilon_r |_{\uo} $ on $M$.
\end{enumerate}
In particular, by $(3)$ the action $\epsilon_r |_{\uo}$ is
integrable, i.e. this $\uo$-action is locally nilpotent. By
\ref{equivariant equvalence 4} and Lemma \ref{general D
equivariance iso} \emph{ii)} (applied to $X = G$ and $K = R$) we
have an equivalence
\begin{equation}\label{equivariant equvalence 5}
\pi^\Po_{G*}( \ )^{P}: \MOD(\D_{G},P, \uo) \leftrightarrows
\MOD(\DtP): \pi^{\Po*}_{G}.
\end{equation}
Note that the functor on the left hand side (that corresponds to)
the global section functor is the functor of taking
$P$-invariants.

Let $\wt{M}_P := \Ug/\Ug \cdot \uo$ be a sort of ``$P$-universal"
Verma module for $\Ug$ and equip it with the $P$-action that is
induced from the right adjoint action of $P$ on $\Ug$. Note that
the object $\BGG_G \ot \epsilon_r(\wt{M}_P) \in \MOD(\D_{G},P,
\uo)$ represents global sections and therefore
corresponds to $\DtP \in \MOD(\DtP)$.

\medskip

Our next task is to describe the (full) subcategories $\MOD(\D_G,P,\uo,
{\lambda})$ and $\MOD(\D_G,P,\uo,
\w{{\lambda}})$ of $\MOD(\D_{G},P, \uo)$
corresponding to the subcategories $\MOD(\DPl)$ and $\MOD^{\w{\lambda}}(\DtP)$ of $\MOD(\DtP)$, respectively.

Let us consider the smash product $\D_G \ast \U(\lo)$ of $\D_G$
and $\U(\lo)$ with respect to the adjoint action of $\lo$ on $\g$.
Thus, $\D_G \ast \U(\lo) = \D_G \ot \U(\lo)$ as a vector space and
its (associative) multiplication is defined by
$$ y \ot x \cdot y' \ot x' := y[\epsilon_r(x),y']\ot x' + yy' \ot xx',\; x \in
\lo, x' \in \U(\lo),y,y' \in \D_G.
$$
Observe that a $(\D_G,L)$-module is the same thing as a $\D_G
\ast \U(\lo)$-module on which the action of $1 \ot \lo$ is integrable (i.e.
its the differential of the given $L$-action). We have an
algebra isomorphism
$$
\D_G \ot \U(\lo) \isoto \D_G \ast \U(\lo), \; y \ot 1 \mapsto y \ot
1, \, 1 \ot x \mapsto 1 \ot x - \epsilon_r(x) \ot 1, \; y \in
\D_G, x \in \lo.
$$
This restricts to
the algebra homomorphism
\begin{equation}\label{whole Ulp}
{\alpha}_{\lo}: \U(\lo) \to \D_G \ast \U(\lo), \;  1 \ot x
\mapsto 1 \ot x - \epsilon_r(x) \ot 1, \; x \in \lo.
\end{equation}
Note that the algebra \textit{anti}-isomorphism ${}^*: \U(\lo) \isoto \U(\lo)$, $x \mapsto - x$, for $x \in \lo$, restricts to an isomorphism
${}^*:\Z(\lo) \isoto \Z(\lo)$.  
\begin{proposition}\label{equivariance condition number 4} \textbf{i)} Let $M  \in \MOD(\DtP)$ and $z \in \Z(\lo)$. Since $\epsilon_l(z) \in
\Z(\lo)=\Z(\DtP)$ it defines a morphism $\epsilon_l(z): M \to M$.
By functoriality we get a morphism $\pi^{\Po*}_G(\epsilon_l(z)):
\pi^{\Po*}_G(M) \to \pi^{\Po*}_G(M)$. We have
$\pi^{\Po*}_G(\epsilon_l(z)) =
\alpha_{\lo}(z^*)|_{\pi^{\Po*}_G(M)}$.

\smallskip \noindent \textbf{ii)} Let $M  \in \MOD(\D_{G},P, \uo)$.
Then $M \in \MOD(\D_G,P,\uo, {\lambda})$ iff the following holds:
 \begin{enumerate}[\hbox{(4)}]
 \item $(\alpha_{\lo}(z^*) - \chi_{\lo,\lambda}
 (z))m = 0, \; m \in M, z \in \mathcal{Z}(\lo)$.
\end{enumerate}
\smallskip \textbf{iii)} Let $M \in \MOD(\D_{G},P, \uo)$.
Then $M \in \MOD(\D_G,P,\uo, {\w\lambda})$  iff the following holds:
\begin{enumerate}[($\w{4}$)]
 \item  $\alpha_{\lo}(z) - \chi_{\lo,\lambda}
 (z)
\hbox{ is locally nilpotent on } M, \hbox{ for } z \in
\mathcal{Z}(\lo).$
\end{enumerate}

\end{proposition}
\begin{proof}
\textbf{i)}. We have $\pi^{\Po*}_G(M) = \BGG_G \ot_{\pi^{\Po-1}_G(\BGG_\Po)}
\pi^{\Po-1}_G(M)$. Let $f \in \BGG_G$ and $m \in
\pi^{\Po-1}_G(M)$. Then for $x \in \lo$ we have $d\rho(x)m = 0$
and consequently
$$
{\alpha}_{\lo}(-x)(f \ot m) = (\epsilon_r(x) -  d\rho(x))(f
\ot m) = f \ot \epsilon_r(x)m.
$$
Since ${\alpha}_{\lo}$ is an algebra homomorphism we get for
$z \in \Z(\lo)$ that
$$\alpha_{\lo}(z^*)(f \ot m) = f \ot
\epsilon_r(z)m = \pi^{\Po*}_G(\epsilon_l(z))(f \ot m).$$ This proves \textbf{i)}.
\textbf{ii)} follows from \textbf{i)}. \textbf{iii)} is similar to \textbf{ii)} and left to the reader.
\end{proof}

Let $M_{P,\lambda} := \Ug/\Ug \cdot (\uo+ \Ker
\chi_{\lo,\lambda})$ be a left $\Ug$-module equipped with the
right $P$-action that is induced from the adjoint action of $P$ on
$\Ug$. Note that the object $\BGG_G \ot
\epsilon_r({M}_{P,\lambda})$ of $\MOD(\D_G,P,\uo,
{{\lambda}})$ represents global sections (=
taking $P$-invariants) and therefore corresponds to $\DPl \in
\MOD(\DPl)$.

\begin{rem}\label{traditional comp rem} Note that when $\lo = \h$ condition $(4)$ becomes the
traditional condition of \cite{BB93}: $\epsilon_r(x)m-d\rho(x)m =
\lambda(x)m$, for $x \in \h$, $m \in M$.
\end{rem}

\begin{rem}\label{other equi conditions rem} Assume that $M \in
\MOD(\D_G,P,\uo)$. Then condition $(4)$ holds for $M$ $\iff$
\begin{enumerate}[($4'$)] \item $(\epsilon_r(z) - \chi_{\lo,\lambda}
 (z))m = 0$, for $m \in M^{L}, z \in \mathcal{Z}(\lo)$.
 \end{enumerate} (Because $(4')$ is obviously equivalent to
 $(\pi^\Po_{G*}M)^{L} \in \MOD(\DPl)$.)

 If we consider $M^{L}$ as a sheaf on $G/L$ it global sections equal $\Gamma_{G}(M)^{L}$,  where  $\Gamma_G(M)$ is the $\BGG(G)$-module corresponding to the $\BGG_G$-module
 $M$.  Since $L$ is reductive $G/L$ is affine,
\cite{Mat60},  and therefore we may replace $M^L$ by $\Gamma_G(M)^L$ in $(4')$.

However, condition $(4)$ is better to work with
then $(4')$,  particularly  while considering modules with an additional equivariance condition from the left side, see Section \ref{section O}.
\end{rem}

\begin{example}\label{G=P equivariance} Let us consider the simplest case when $P = G$. Then $\uo = 0$
and we write $\MOD(\D_G,G, {\lambda}) := \MOD(\D_G,G,\uo_G,
{\lambda})$ for simplicity.

The equivalence $\MOD(\C)\cong \MOD(\BGG_G,G)$, $V
\mapsto \BGG_G \ot V$, induces for any $\lambda \in \h^*$ the
equivalence $\MOD(\Ugl) \cong \MOD(\D_G,G,
{\lambda})$ given by
$$
V \mapsto \BGG_G \otimes V
$$
where $(\BGG_G \otimes V)^G = V$ is a left module for
$\epsilon_l(\Ug)^{\lambda}$. Similarly with $\chi_\lambda$
replaced by $\w{\chi_\lambda}$.
\end{example}
\begin{example}\label{P=B equivariance} Let $P = B$. Let $\lambda
\in \h^*$ and let $M_\lambda$ be the Verma module for
$\epsilon_r(\Ug)$ with highest weight $\lambda$. Let $\mu \in
\h^*$ be integral. Consider the algebraic $B$-action
$\rho$ on $M_\lambda$ which after differentiation satisfies
$$
d\rho(x)m = (x-\lambda(x)+\mu(x))m, \, m \in M_\lambda, x \in \bo.
$$
Denote by $M_{\lambda,\mu}$ the Verma module $M_\lambda$ equipped
with this $B$-action. Then we have that
$$
\BGG_G \otimes M_{\lambda,\mu} \in \MOD(\D_G, B, \n, \lambda-\mu).
$$
For $\mu = 0$ we have mentioned that the functor $\Hom_{\MOD(\D_G, B,
\n, \lambda)}(\BGG_G \otimes M_{\lambda,0} , \ )$ is naturally
equivalent to the global section functor on $\MOD(\D_G, B, \n,
\lambda)$, so that $\BGG_G \otimes M_{\lambda,0} \cong
\pi^{G*}_\B \DBl$. This implies
\begin{equation}\label{global Verma}
\End_{\MOD(\D_G, B, \n, \lambda)}(\BGG_G \otimes M_{\lambda}) =
\Gamma(\DBl) = \Ugl.
\end{equation}
To get an idea of a general $\BGG_G \otimes M_{\lambda,\mu}$
assume for instance that $\mu \geq 0$. Then there is an injective
map
\begin{equation}\label{inj1map}
f: \BGG_G \otimes M_{\lambda,\mu} \to \BGG_G \otimes
M_{\lambda-\mu,0}.
\end{equation}

By the Peter-Weyl theorem $\BGG_G \cong \oplus_{\phi \in
\Lambda_+} V^*_G(\phi) \ot V_G(\phi)$ as a $G$-bimodule. Let
$v_\phi \in V_G(\phi)$ be a highest weight vector. Let $1_\lambda$
and $1_{\lambda-\mu}$ be highest weight vectors in $M_{\lambda,
\mu}$ and $M_{\lambda-\mu,0}$, respectively. We can define $f$ by
$f(1 \ot 1_\lambda) := (v \ot v_\mu) \ot 1_{\lambda-\mu}$ where $v
\in V^*_G(\mu)$ is any non-zero vector. $f$ is injective since
both sides of \ref{inj1map} are free over the integral domain
$\BGG_G \ot \epsilon_r(\U(\n_-))$. Note that $f$ is not an
isomorphism (and the two objects of \ref{inj1map} must be
non-isomorphic) unless $\mu = 0$.

\end{example}
\subsection{Global sections}

\noindent The left $G$-action on $G/R$, $(g,
\overline{g'}) \mapsto \overline{gg'}$, commutes with the right
$L$-action and therefore induces a homomorphism $\Ug \to \DtP$.
There is also the map $\epsilon: S(\h)^{\W_P} = \Z(\lo) \to \DtP$.
These maps agree on $S(\h)^\W$ and hence induces a map
$$
\tU^{\W_P} = \U \ot_{\Z} S(\h)^{\W_P} \to \DtP.
$$
This induces a homomorphism $\Ul = \tU^{\W_P}/(I
_{\lo,\lambda}) \to \DPl$.

Consider the sheaf of algebras $\BGG_\Po \otimes \Ug$ on $\Po$
with multiplication determined by those in $\BGG_\Po$ and in $\Ug$
and by the requirement that $[A,f] = \epsilon(A)(f)$ for $A \in
\g$ and $f \in \BGG_\Po$. Then we have a surjective algebra
homomorphism $\eta: \BGG_\Po \otimes \Ug \to \DtP$. Its kernel is
the ideal generated by $\xi \in \BGG_\Po \otimes \uo$, $\xi(x) \in
\p_x$, for $x \in \Po$ and $\p_x \subseteq \g$ the corresponding
parabolic subalgebra.

Hence, to define a $\DtP$-module structure on an $\BGG_\Po$-module
$M$ is the same thing as defining a $\Ug$-module structure on $M$
such that $\Ker \eta$ vanishes on $M$ and $A(fm) = f(Am) +
\epsilon(A)(f)m$, for $A \in \g$, $f \in \BGG_\Po$ and $m \in M$.

 Let $\mu \in \h^*$ be integral and
$P$-dominant. Recall that $V_{P}(\mu)$ denotes the corresponding
irreducible representation of $P$ with highest weight $\mu$ and
$\BGG(V_{P}(\mu))$ the corresponding left $G$-equivariant locally
free sheaf on $\Po$.

Let $M \in \MOD(\DtP)$. We shall show that the $\BGG_\Po$-module
$M\ot_{\BGG_\Po} \BGG(V_{P}(\mu))$ is naturally a $\DtP$-module.
We proceed as follows:

The $G$-action on $\BGG(V_{P}(\mu))$ differentiates to a left
$\g$-action on it, which extends to a $\g$-action on
$M\ot_{\BGG_\Po} \BGG(V_{P}(\mu))$ by Leibniz's rule. Since
$V_{P}(\mu)$ is an irreducible $P$-module we have that $R$ acts
trivially on it (recall $V_P(\mu) = V_{L}(\mu)$). Hence, $\uo$
acts trivially $\BGG(V_{P}(\mu))$ and from this it now follows
that the compatibilities for being a $\DtP$-module are satisfied
by $M\ot_{\BGG_\Po} \BGG(V_{P}(\mu))$.

Assume that $M \in \MOD(\DtP)$. In the equivariant language on $G$
we see that $M$ and $M \ot_{\BGG_\Po} \BGG(V_{P}(\mu))$ correspond
to $\pi^{\Po*}_GM$ and  $M_{V_P(\mu)} := (\pi^{\Po*}_G M) \otimes
V_{P}(\mu) \in \MOD(\D_{G}, P, \uo)$, respectively. Here,
the $\D_G$-action on $M_{V_P(\mu)}$ is given by the action on the
first factor and the $P$-action is diagonal. Again, it is the fact
that $R$ acts trivially on $V_{P}(\mu)$ that shows that
$M_{V_P(\mu)}$ is an object of $\MOD(\D_{G}, L, \uo)$.

\begin{lem}\label{Lemacita}
Let $\lambda \in \h^*$, $M \in \MOD(\DPl)$ and $\mu \in \h^*$ be
integral and $P$-dominant. Then $M\ot_{\BGG_\Po} \BGG(V_P(\mu))
\in \oplus_{ \nu\in \Lambda(V_P(\mu))} \MOD^{\w{\lambda
+ \nu}}(\DtP)$, where $\Lambda(V_P(\mu))$ denotes the set of weights of
$V_P(\mu)$.
\end{lem}
\begin{proof} In equivariant
translation we want to prove that
\begin{equation}\label{to prove here}
M_{V_P(\mu)}  \in \oplus_{ \nu\in \Lambda(V_P(\mu))}\MOD(\D_{G}, P,
\uo, {\w{{\lambda + \nu}}}).
\end{equation}
We use Proposition \ref{equivariance condition number 4}
\textbf{i)}. We have an action $\wt{\alpha}_{\lo}: \U(\lo) \to
\End(M_{V_P(\mu)})$. We see that this action is actually the
tensor product of the $\wt{\alpha}_{\lo}$-action of $\U(\lo)$ on
$\pi^{\Po*}_GM$ and the $\U(\lo)$-action on $V_{P}(\mu)$, which is
the differential of the given $L$-action. Now, since for $z \in
\Z(\lo)$, we by assumption have that $\alpha_{\lo}(z) =
\wt{\alpha}_{\lo}(z)$ acts by $\chi_{\lo,\lambda}(z)$ on
$\pi^{\Po*}_GM$ it follows from \cite{BerGel} that \ref{to prove
here} holds.
\end{proof}

\begin{thm}\label{thmtag0}   $i)$ $R{\pPB}_*\DtB = \DtP \ot _{{\Z(\lo)}} S(\h)$,
$ii)$  $R{\pQP}_*\DtP=  \DtQ \ot _{{\Z(\lo_Q)}} S(\h)^{\W_P}$,
$iii)$ $R\Gamma(\DtP) = \tU{}^{\W_P}$ and $iv)$ $R\Gamma(\DPl) =
\Ugl$ .
\end{thm}
\begin{proof}
By Lemma \ref{classicalglobal} and Lemma
\ref{relativeclassicalglobal} the associated graded maps $i)$ and
$ii)$ are isomorphisms; hence $i)$ and $ii)$ are also
isomorphisms. $iii)$ is a special case of $ii)$ and $iv)$ follows
from $iii)$ because $R\Gamma$ commutes with $( \ ) \ot_{\Z(\lo)}
\C_\lambda$, since $\DtP$ is locally free over $\Z(\lo)$.
\end{proof}

The functor  $\Gamma: \MOD(\DPl) \to
\MOD(\Ugl)$ has a left adjoint  $\Loc := \DPl \otimes_{\Ugl} ( \ )$,  called the localization functor.
Also $\Gamma: \MOD^{\w{\lambda}}(\DtP) \to
\MOD^{\w{\lambda}}(\U)$  has a left adjoint
$\Loc := \underleftarrow{\lim}_n \D_\Po/(I_\lambda)^n
\otimes_{\Ug} ( \ )$.

\section{Singular Localization}
Here we prove the singular version of Beilinson-Bernstein
localization.
\begin{thm}\label{singlocthmtag0} Let $\lambda$ be dominant and $P$-regular then $\Gamma: \MOD(\DPl) \to \MOD(\Ugl)$
is an equivalence of categories. \end{thm}
\begin{proof}
Essentially taken from \cite{BB81}. Since $\Gamma$ has a left adjoint $\mathcal{L}$ which is right exact and since $\Gamma \circ \mathcal{L}(\Ugl) = \Gamma(\DPl) = \Ugl$,
the theorem will
follow from the following two claims:

$a)$ Let $\lambda$ be dominant. Then $\Gamma: \MOD(\DPl) \to
\MOD(\Ugl)$ is exact.

 $b)$ Let $\lambda$ be dominant and $P$-regular and $M \in \MOD(\DPl)$, then if $\Gamma(M) = 0$ it follows that $M = 0$.

Let $V$ be a finite dimensional irreducible $G$-module and let
$$
0 = V_{-1} \subset V_{0} \subset \ldots \subset V_n = V
$$
be a filtration of $V$ by $P$-submodules, such that $V_i/V_{i-1}
\cong V_P(\mu_i)$ is an irreducible $P$-module.

We first chose $V$ so that its highest weight $\mu_0$  is a
$P$-character. Thus $M \ot_\BGG \BGG(V_0) = M(-\mu_0)$ and we get
an embedding $M(-\mu_0) \hookrightarrow M \ot_\BGG \BGG(V)$, which
twists to the embedding $M   \hookrightarrow M(\mu_0) \ot_\BGG
\BGG(V) \cong M(\mu_0)^{\dim V}$. Now, by Lemmas \ref{regdomlem},
\ref{Lemacita} and Theorem \ref{thmtag0} iii) we get that this
inclusion splits on derived global sections, so $R\Gamma(M)$ is a
direct summand of $R\Gamma(M(\mu_0))^{\dim V}$. Now, for $\mu_0$
big enough and if $M$ is $\BGG$-coherent we have
$R^{>0}\Gamma(M(\mu_0)) = 0$ (since $\BGG(\mu_0)$ is very ample).
Hence, $R^{>0}\Gamma(M) = 0$ in this case. A general $M$ is the
union of coherent submodules and by a standard limit-argument it
follows that $R^{>0}\Gamma(M) = 0$. This proves $a)$.

Now, for $b)$ we assume instead that the lowest weight $\mu_n$ of
$V$ is a $P$-character. Then we have a surjection $M^{\dim V}
\cong M\otimes_\BGG \BGG(V) \to M(-\mu_n)$. Applying global
sections and using Lemmas \ref{Pregularweightlemma},
\ref{Lemacita} and Theorem \ref{thmtag0} iv) we get that
$\Gamma(M(-\mu_n))$ is a direct summand of $\Gamma(M)^{\dim V}$.
For $\mu_n$ small enough we get that $\Gamma(M(-\mu_n)) \neq 0$.
Hence, $\Gamma(M) \neq 0$. This proves $b)$.
\end{proof}

Assume that $\lambda$ is $P$-regular. Then the projection  $\h^*/{\W_P} \to \h^*/\W$ is unramified at $\lambda$ and from this one deduces, see \cite{BeGin99},
that restriction defines an equivalence of categories  $\MOD^{\w{\lambda}}(\tU{}^{\W_P}) \isoto \MOD^{\w{\lambda}}(\U)$.
\begin{thm}\label{singlocthmtagGeneralized0} Let $\lambda$ be dominant and $P$-regular then
$\Gamma: \MOD^{\w{\lambda}} (\DtP) \to
\MOD^{\w{\lambda}}(\tU{}^{\W_P}) \cong
\MOD^{\w{\lambda}}(\U)$ is an equivalence of categories.
\end{thm}
\begin{proof} This follows from Theorem \ref{singlocthmtag0} and
a simple devissage.
\end{proof}

\section{Translation functors}\label{Translation functors}
We geometrically describe translations functors on $\g$-modules in the context
of singular localization. For regular localization this was worked out in \cite{BeGin99}.
Singular localization clarifies the picture. We get one-one correspondences  between translation functors and geometric functors and
all global section functors can be made to take values in $\MOD(\U)$. Thus ramified coverings of the form $\h^*/\W_\lambda \to \h^*/\W_\mu$
will not complicate the picture
as they appeared to do in \cite{BeGin99}.

\subsection{Translation functors}\label{translationsec}
For any $\Z(\lo)$-algebra $S$ let
$\MOD^{\Z(\lo)\text{-}\fin}(S)$ be the category of $S$-modules that are locally finite over $\Z(\lo)$. Thus $\MOD^{\Z(\lo)\text{-}\fin}(S) = \oplus_{\mu \in \h^*} \MOD^{\w\lambda}(S)$ and we have exact projections
$pr_{\lo, \w\mu}: \MOD^{\Z(\lo)\text{-}\fin}(S) \to \MOD^{\w\mu}(S)$. We put  $pr_{\w\mu} := pr_{\g, \w\mu}$.

Assume $\lambda, \mu \in \h^*$ satisfy $\lambda- \mu$ is integral. Then there is the translation functor
$$T^\mu_{\lo,\lambda} : \MOD^{\w{\lambda}}(\U(\lo)) \to \MOD^{\w{\mu}}(\U(\lo)), \, M \mapsto pr_{\lo, \w\mu}(M \otimes E)
$$
where $E$ is an irreducible finite dimensional representation of
$\lo$ with extremal weight $\mu - \lambda$. Again, put $T^\mu_{\lambda} := T^\mu_{\g,\lambda}$.
See \cite{BerGel} for further information
about translation  functors.

We shall give a $\D$-module interpretation of these functors.  We use the language of $\DtP$-modules; it is a simple task to pass to an equivariant description on
$G$.
Define for any parabolic subgroup $P \subset G$ a geometric
translation functor
$$\Tt^\mu_{P,\lambda}: \MOD^{\w{\lambda}}(\DtP) \to \MOD^{\w{\mu}}(\DtP),
\; M \mapsto pr_{\lo, \w\mu}(M\ot_{\BGG_\Po} \BGG(E))$$ for $M
\in \MOD^{\w{\lambda}} (\DtP)$, where $E$ is an irreducible
$P$-representation with highest weight in $\W_P(\mu-\lambda)$.

Note that if $\mu-\lambda$ is a $P$-character then $\BGG_\Po(E) =
\BGG_\Po(\mu-\lambda)$ and in this case $\Tt^\mu_{P,\lambda} = ( \ )
\ot_{\BGG_\Po} \BGG(\mu-\lambda)$ is an equivalence with inverse
given by $\Tt^\lambda_{P,\nu} = ( \ )\ot_{\BGG_\Po}
\BGG(\lambda-\mu)$. In particular, for $P=B$ we have
$\Tt^\mu_{B,\lambda} = ( \ ) \ot_{\BGG_\B} \BGG(\mu-\lambda)$ for any
$\mu$ and $\lambda$.

Let $Q \subset G$ be another parabolic subgroup with $P \subset
Q$. We have
\begin{lem}\label{commutativity of translation} The diagram
$$
   \begin{diagram}
    \node[0]{\MOD^{\w{\lambda}}(\DtP)} \arrow{e,l}{\Tt^\mu_{P,\lambda}} \arrow{s,r}{\pQPls} \node[0]{\MOD^{\w{\mu}}(\DtP)}
\arrow{s,r}{\pQPls} \\
    \node[0]{\MOD^{\w{\lambda}}(\DtQ)}
    \arrow{e,l}{\Tt^\mu_{Q,\lambda}} \node{\MOD^{\w{\mu}}(\DtQ)}
   \end{diagram}
$$
of exact functors commutes up to natural equivalence.
\end{lem}
In the case of $P=B$ and $Q=G$ this was proved in \cite{BeGin99}.
\begin{proof} Let $V$ (resp., $V'$) be an irreducible finite dimensional representation for $Q$ (resp., for $P$)
whose highest weight belongs to $\W_Q(\mu-\lambda)$ (resp.,
$\W_P(\mu-\lambda)$). Let $M \in \MOD^{\w{\lambda}}(\DtP)$.
Then, since $V$ is a $Q$-representation, we have $\BGG_\Po(V) =
\pQPus(\BGG_\Q(V))$ and therefore it follows from the projection
formula that
$$ \pQPls(\BGG_\Po(V) \otimes_{\BGG_\Po} M) = \BGG_\Q(V) \ot_{\BGG_\Q} \pQPls(M).$$
Thus we get
$$
\Tt^\mu_{Q,\lambda} \circ \pQPls (M) =
pr_{\lo_Q, \w\mu}(\BGG_\Q(V) \ot_{\BGG_\Q} \pQPls(M))  =
$$
$$
pr_{\lo_Q,\w\mu}(\pQPls(\BGG_\Po(V) \otimes_{\BGG_\Po} M)) =
\pQPls( pr_{\lo,\w\mu}(\BGG_\Po(V) \otimes_{\BGG_\Po} M))
\overset{(*)}{=}
$$
$$
\pQPls( pr_{\lo,\w\mu}(\BGG_\Po(V') \otimes_{\BGG_\Po} M)) =
\pQPls \circ \Tt^\mu_{P,\lambda} (M).
$$
The equality $(*)$ follows from Lemma \ref{Pregularweightlemma}
applied to the reductive Lie algebra $\lo_Q$ and its parabolic
subalgebra $\lo_Q \cap \p$ (compare with the proof of the
localization theorem).
\end{proof}

\smallskip

\noindent Let us geometrically describe \textit{translation to the
wall}: In this case $\Delta_\lambda \subsetneq \Delta_\mu$. We assume that
$\lambda$ and $\mu$ are dominant. We choose the parabolic
subgroups $P \subset Q \subset G$ such that the parabolic roots of
$P$ equal $\Delta_\lambda$ and the parabolic roots of $Q$ equal
$\Delta_\mu$. By Theorem \ref{singlocthmtagGeneralized0} and Lemma \ref{commutativity of
translation} it follows that the diagram below commutes up to natural
equivalence:
  \begin{equation}\label{translation diagram to the wall}
   \begin{diagram}
    \node[0]{\MOD^{\w{\lambda}}(\U)}\arrow[2]{s,r}{(4) \; T^\mu_{\lambda}} \node[0]{\MOD^{\w{\lambda}}(\DtP)} \arrow{w,l}{(1) \; \Gamma}
\arrow{s,r}{(3) \;\pQPls}  \arrow{se,t}{(2) \; \Tt^\mu_{P,
\lambda}}  \\
    \node[2]{\MOD^{\w{\lambda}}(\DtQ)} \arrow{s,r}{{(5) \; \Tt^\mu_{Q,
\lambda}}} \node{{\MOD^{\w{\mu}}(\DtP)}}   \arrow{sw,b}{(7)\; \pQPls}  \\
     \node[0]{\MOD^{\w{\mu}}(\U)} \node[0]{\MOD^{\w{\mu}}(\DtQ)} \arrow{w,l}{(6) \; \Gamma}
   \end{diagram}
  \end{equation}
Note that $(1)$ and $(6)$ are equivalences by the choices of $P$
and $Q$ and that $(2) = ( \ ) \otimes_{\BGG_\Po}
\BGG(\mu-\lambda)$ is an equivalence, since $\mu - \lambda$ is a
$P$-character.

We see that $(3)$ is an equivalence of categories because both the
source and the target categories are D-affine, since $\lambda$ is
$P$- and $Q$-regular, and $\Gamma \circ\pQPls = \Gamma$. On the
other hand, the functor $(7)$ is not faithful, because $\mu$ is
not $P$-regular. $(5)$ is also not faithful. We remind that all
functors involved are exact.

\smallskip

\noindent Let us now describe \textit{translation out of the
wall}: This is done by taking the diagram of adjoint functors in
the diagram \ref{translation diagram to the wall}, so we keep
assuming that $\lambda$, $\mu$, $P$ and $Q$ are as in
\ref{translation diagram to the wall}. The left and right adjoint
of $T^\mu_{\lambda}$ is $T^\lambda_{\mu}$, the translation out of
the wall. The equivalences $(1), (2), (3)$ and $(6)$ of course
have left and right adjoints that coincide. Also, the left and
right adjoint of $(5)$ coincide; it is given by
$\Tt^\lambda_{Q, \mu}$. Finally $(7)$ has the left adjoint
$\pQPus$; thus, $\pQPus$ must also be the right adjoint of $(7)$.
Summing up we have:
  \begin{equation}\label{translation out of the wall}
   \begin{diagram}
    \node[0]{\MOD^{\w{\lambda}}(\U)}  \arrow{e,l}{\Loc} \node[0]{\MOD^{\w{\lambda}}(\DtP)} \\
    \node[2]{\MOD^{\w{\lambda}}(\DtQ)} \arrow{n,r}{\pQPus}   \node{{\MOD^{\w{\mu}}(\DtP)}}
    \arrow{nw,t}{\Tt^\lambda_{P,
\mu}}  \\
     \node[0]{\MOD^{\w{\mu}}(\U)}  \arrow[2]{n,r}{T^\lambda_{\mu}} \arrow{e,l}{\Loc}
     \node[0]{\MOD^{\w{\mu}}(\DtQ)}\arrow{n,r}{{\Tt^\lambda_{Q,
\mu}}}  \arrow{ne,b}{\pQPus}
   \end{diagram}
  \end{equation}

\section{Category $\BGGcat$ and Harish-Chandra (bi-)modules.}\label{section O}
Singular localization allows us to interpret blocks of category $\BGGcat$ as bi-equivariant $\D_G$-modules
which in turn are equivalent to categories of  Harish-Chandra (bi-)modules. As we mentioned in the introduction, the novelty here is that we are lead to consider $\g\text{-}\lo$-bimodules, which we believe is a better notion.
\textit{Parabolic} (and singular) blocks of $\BGGcat$ are discussed in Section \ref{Whittaker two}.

The material here is related to Section \ref{Translation functors}  because translation functors restrict to functors between blocks in
$\BGGcat$.
\subsection{Category $\BGGcat$ and generalized twisted Harish-Chandra modules.} See \cite{Hum08} for generalities on category $\BGGcat$ and
\cite{Dix77} for generalities on Harish-Chandra modules.

\medskip

\noindent We are interested in the Bernstein-Gelfand-Gefand
category $\BGGcat$ of finitely generated left $\Ug$-modules which
are locally finite over $\U(\n)$ and semi-simple over $\h$.
For $\lambda \in \h^*$ we let $\BGGcat_{\lambda},
\BGGcat_{\w{\lambda}} \subset \BGGcat$ be the subcategories
of modules with central character, respectively, generalized
central character, $\chi_\lambda$.

\noindent \emph{Generalized twisted Harish-Chandra modules.} Let
$K \subset G$ be a subgroup and let $\mathfrak{k} := Lie \, K$ be
its Lie algebra. A \emph{weak Harish-Chandra} $(K, \Ug)$-module
(or simply a $(K,\Ug)$-module) is a left $\Ug$-module $M$ equipped
with an algebraic left action of $K$ such that the action map $\Ug
\ot M \to M$ is $K$-equivariant with respect to the adjoint action
of $K$ on $\Ug$. A \emph{Harish-Chandra} $(K, \Ug)$-module (or
simply a $(\mathfrak{k},K,\Ug)$-module) is a weak Harish-Chandra
module such that the differential of the $K$-action coincides with
the action of $\mathfrak{k} \subset \Ug$.

Similarly, there are $(K, \Ug^\lambda)$-modules and
$(\mathfrak{k}, K, \Ug^\lambda)$-modules, for $\lambda \in \h^*$.

Let $\mu \in K^*$. A $\mu$-\emph{twisted Harish-Chandra} module is
a $(K, \Ug)$-module $M$ on which the action of $\mathfrak{k}
\subset \Ug$ minus the differential of the $K$-action is equal to
$\mu$.

We shall now give certain generalizations of twisted
Harish-Chandra modules in the case when $K = P$. Consider the
smash-product algebra $\U \ast \U(\lo)$ with respect to the
adjoint action of $\lo$ on $\Ug$. Observe that an $(L, \U)$-module
is the same thing as a $\U \ast \U(\lo)$-module on which $1 \ot
\lo$ acts semi-simply and $1 \ot H_\alpha$ has integral
eigenvalues for each simple coroot $H_\alpha$. The algebra
\emph{anti-homomorphism} $\U(\lo) \to \U \ast \U(\lo)$, defined by $x \mapsto x
\ot 1 - 1 \ot x$, for $x \in \lo$, restricts to a \emph{homomorphism}
\begin{equation}\label{alphabar}
\overline{\alpha}_{\lo}: \Z(\lo) \to \Z(U(\g) \ast \U(\lo)).
\end{equation}
(Compare with the map $\alpha_{\lo}(z^*)$ from \ref{whole Ulp}.) We
define $\MOD(\w{{\lambda}}, \uo, P,
\Ug^{\lambda'})$ to be the category of $(P,
\Ug^{\lambda'})$-modules $M$ such that, if $\rho$ denotes the
$P$-action on $M$, then $d\rho |_{\uo}$ coincides with the action
of $\uo \subset \Ug^{\lambda'}$ on $M$ and for $z \in \Z(\lo)$ we
have that $\overline{\alpha}_{\lo}(z) - \chi_{\lo,\lambda}(z)$
acts locally nilpotently on $M$.

Similarly, one defines categories
$\MOD^{\w \lambda'}(\w{{\lambda}}, \uo, P, \Ug)$
and $\MOD({{\lambda}}, \uo, P, \Ug^{\lambda'})$, etc.

\medskip

We see that if $\lambda, \lambda' \in \h^*$, $\lambda-\lambda'$ is
integral then
$$
\BGGcat_\lambda = \mmOD(\lambda',\n, B,\Ugl) \hbox{ and }
\BGGcat_{\w{\lambda}} =\mmOD^{\w{\lambda}}(\lambda',\n, B,\Ug
)
$$
are (non-generalized) categories of twisted Harish-Chandra
modules. For $P \neq B$ we like to think of
$\mmOD(\w{{\lambda}}, \uo, P, \Ug^{\lambda'})$ and
$\mmOD({\lambda}, \uo, P, \Ug^{\lambda'})$ as
``non-standard parabolic blocks in $\BGGcat$" although, in
reality, they are not even subcategories of $\BGGcat$, since the
$\bo$-action is not locally finite.

\subsection{Harish-Chandra modules to bimodules} The categories of the previous section can be described in terms of
Harish-Chandra bimodules, \cite{BerGel}. Let $\wt{\Hcal}(\lo)$ be the category of $\U\text{-}\U(\lo)$-bimodules on which the adjoint action of $\lo$ is integrable and
the left action of $\uo$ is locally nilpotent. Write $\wt{\Hcal} := \wt{\Hcal}(\g)$ and replacing $\g$ by $\lo$ we write
$\wt{\Hcal}(\lo,\lo)$ for the category of $\U(\lo)\text{-}\U(\lo)$-bimodules on which the adjoint $\lo$-action is integrable.

Let $\Hcal(\lo) \subset \wt{\Hcal}(\lo)$ be the subcategory of noetherian objects. Note that for $M \in \wt{\Hcal}(\lo)$ we have
$M \in \Hcal(\lo)$ $\iff$ $M$ is f.g. as a $\U\text{-}\U(\lo)$-bimodule $\iff$ $M$ is f.g. as a left $\U$-module (and in case $\lo = \g$ this holds if and only if $M$ is f.g. as a right $\U$-module). Put
$$ {}_{\Zfin} \Hcal(\lo) := \{M \in   \Hcal(\lo); \;  \Z \hbox{ acts locally finitely on } M \hbox{ from the left} \},
$$
$$
\Hcal(\lo)_{\Zlfin}  :=   \{M \in   \Hcal(\lo); \;  \Z(\lo) \hbox{ acts locally finitely on } M \hbox{ from the right} \}
$$
and  $ {}_{\Zfin} \Hcal(\lo)_{\Zlfin} := {}_{\Zfin} \Hcal(\lo)  \cap \Hcal(\lo)_{\Zlfin}$. Observe that
\begin{equation}\label{center both sides}
 {}_{\Zfin} \Hcal =    \Hcal_{\Zfin} =   {}_{\Zfin} \Hcal_{\Zfin}.
\end{equation}
We set
${}_{{ \lambda'}} \Hcal(\lo) :=  \{M \in \Hcal(\lo); \; I_{ \lambda'}  M= 0\}$,
 $ \Hcal(\lo)_{ \lambda} :=  \{M \in \Hcal(\lo); \; M I_{\lo, \lambda}  = 0\}$ and
 ${}_{\w{\mu}} \Hcal(\lo) :=  \{M \in \Hcal(\lo); \; I_{ \lambda'} \hbox{ acts locally nilpotently on $M$} \}$, etc. Similarly, we define
 ${}_{ { \lambda'}} \Hcal(\lo)_{\w{\lambda}} :=  {}_{{ \lambda'}} \Hcal(\lo)  \cap \Hcal(\lo)_{\w{\lambda}}$,  $\wt{\Hcal}(\lo)_\lambda$, etc.

\begin{lem}\label{aNiceEquivalence} $\MOD({{\lambda}}, \uo, P, \Ug^{\lambda'}) \cong {}_{{ \lambda'}} \Hcal(\lo)_{ {\lambda}} $.
and  $\MOD(\w{{\lambda}}, \uo, P,
\Ug^{\lambda'}) \cong {}_{{\lambda'}} \Hcal(\lo)_{\w{\lambda}} $.
\end{lem}
\begin{proof}  A $(P,\U^{\lambda'})$-module is the same thing as a $\U^{\lambda'} \ast \U(\p)$-module such that $1 \ot \p$ acts integrably.
Under the algebra isomorphism
$$
\U^{\lambda'} \ast \U(\p) \isoto \U^{\lambda'} \ot \U(\p), \; 1 \ot x \mapsto 1 \ot x + x \ot 1, y \ot 1 \mapsto y \ot 1
$$
the latter modules are equivalent to the category of $\U^{\lambda'} \ot \U(\p)$-modules on which
the action of $\Delta \p$ is integrable, where
$\Delta: \p \to \U^{\lambda'} \ot \U(\p)$ is given by $\Delta x := x \ot 1 + 1 \ot x$.

The $\Delta \p$-integrability is equivalent to
$\Delta \lo$-integrability and  that $\Delta \uo$  acts locally
nilpotently. Thus $\MOD(\uo, P, \Ug^{\lambda'})$
is equivalent to the category
of  $\U^{\lambda'} \ot \U(\lo)$-modules such that the action of
$\Delta \lo$ is integrable and $\uo \subset \U^{\lambda'}$ acts
nilpotently. Thus, using the  principal
anti-isomorphism of $\lo$ to identify
 $\U^{\lambda'} \ot \U(\lo)$-modules with  $\U^{\lambda'}\text{-}\U(\lo)$-bimodules, we get $\MOD(\uo, P, \Ug^{\lambda'}) \cong {}_{{ \lambda'}} \Hcal(\lo)$. From this one deduces the lemma.
\end{proof}

\subsection{Bi-equivariant $\D$-modules and category $\BGGcat$}\label{catO} We want to describe
blocks in category $\BGGcat$ in terms of bi-equivariant
$\D_G$-modules. Let $\lambda \in \h^*$. Throughout this
section we assume that $\lambda' \in \h^*$ is a regular dominant
weight such that $\lambda-\lambda'$ is integral.

Denote by
$
\MOD(\lambda',\n,B, \D_G, P, \uo, \w{{\lambda}})
$
the full subcategory of $\MOD(\D_G, P, \uo,
\w{{\lambda}})$ whose object $M$ satisfies
$(1)-(3), (\w{4})$ from Section \ref{Definition of extended
differential operators} and is in addition equipped with a left
$B$-action $\tau: B \to \Aut(M)$ that commutes with $\rho: P \to
\Aut(M)^{op}$ and satisfies
 \begin{enumerate}[(5)]
 \item  $d\tau(x)m = (\epsilon_l(x) -
\lambda'(x))m, \hbox{ for } m \in M, \, x \in \bo.$
\end{enumerate}
(Strictly speaking, $\MOD(\lambda',\n,B, \D_G, P, \uo, \w{{\lambda}})$ is
obtained from $\MOD(\D_G, P, \uo, \w{{\lambda}})$
by adding a $B$-action, but since this $B$-action is determined by
its differential it identifies
with a subcategory of it.)
\begin{lem} Assume that $\lambda$ is $P$-regular. Then $\mmOD(\lambda',\n,B, \D_G, P, \uo, \w{{\lambda}}) \cong {\BGGcat}_{\w{\lambda}}$.
\end{lem}
\begin{proof} We remind that, since $\lambda$ is $P$-regular, restriction defines an equivalence of categories  $res: \MOD^{\w{\lambda}}(\tU{}^{\W_P}) \isoto \MOD^{\w{\lambda}}(\U)$.
Now $(\w{4})$, the two lines preceding it and Theorem
\ref{singlocthmtagGeneralized0} give the equivalence $$\MOD(\D_G, P,
\uo, \w{{\lambda}}) \cong
\MOD^{\w{\lambda}}(\U), \ V \mapsto res(V^P).$$ From
this we deduce that the full subcategory $\BGGcat_{\w{\lambda}}
= \mmOD^{\w{\lambda}}(\lambda',\n,B,\Ug)$ of
$\MOD^{\w{\lambda}}(\U)$ is equivalent to $\mmOD(\lambda',\n,B, \D_G, P, \uo, \w{{\lambda}})$.
\end{proof}

Using the inversion on $G$, left $B$-action and right $P$-action
become right $B$-action and left $P$-action, so $\mmOD(\lambda',\n,B, \D_G, P, \uo, \w{{\lambda}})$  is
equivalent to a full subcategory of $\MOD(\D_G,B,\n, \lambda')$
that we denote by
\begin{equation}\label{equi BGG2}
\mmOD(\w{{\lambda}}, \uo, P, \D_G, B, \n,
\lambda')
\end{equation}
whose definition is obvious. Since $\lambda'$ is dominant and
regular we get from Beilinson-Bernstein localization that
$\MOD(\D_G,B,\n, \lambda') \cong \MOD(\Ug^{\lambda'})$. This induces
an equivalence between \ref{equi BGG2} and
$\mmOD(\w{{\lambda}}, \uo, P, \Ug^{\lambda'})$. (This is not the parabolic-singular Koszul
duality of \cite{BGS}.)

Similarly, if we don't pass to global sections on $\B$, we have
that \ref{equi BGG2} is equivalent to the category
$\mmOD(\w{{\lambda}}, \uo, P, \D^{\lambda'}_\B)$,
whose definition is also obvious.

Summarizing we get
\begin{proposition}\label{singpar} $\BGGcat_{\w{\lambda}} \cong \mmOD(\w{{\lambda}}, \uo, P, \Ug^{\lambda'}) \cong \mmOD(\w{{\lambda}}, \uo, P,
\D^{\lambda'}_\B)$, for $\lambda$ dominant and $P$-regular.
\end{proposition}
Thus, by Lemma \ref{aNiceEquivalence}
\begin{cor}\label{singparcor} $\BGGcat_{\w{\lambda}} \cong {}_{{ \lambda'}} {\Hcal(\lo)}_{\w{\lambda}} $.
\end{cor}
Similarly, one shows that $\BGGcat_\lambda \cong  \mmOD({{\lambda}}, \uo, P, \Ug^{\lambda'}) \cong \mmOD({{\lambda}}, \uo, P,
\D^{\lambda'}_\B) \cong  {}_{{ \lambda'}} {\Hcal(\lo)}_{{\lambda}}$.
\begin{example}\label{Soergel's result example} Let $P=B$ and $\lambda \in \h^*$ be regular and dominant.
Then $\BGGcat_{\w{\lambda}} \cong \mmOD(\w{\lambda}, \n, B, \Ug^{\lambda'})$, which is the category of left
$\Ug^{\lambda'}$-modules which are locally finite over $\bo$ (so
the $\h$-action need not be semi-simple). This equivalence was
first established in \cite{Soe86}.
\end{example}
\begin{example}\label{example HC-bimodules} Let $P = G$ and $\lambda \in \h^*$ be
any weight. Since $\uo_G = 0$ we write for simplicity
$\MOD(\w{{\lambda}}, G, \Ug^{\lambda'}) :=
\MOD(\w{{\lambda}}, \uo_G, G, \Ug^{\lambda'})$. Put
$\BGGcat_{\w{\lambda + \Lambda}} := \oplus_{\mu \in \Lambda}
\BGGcat_{\w{\lambda + \mu}}$. Then we have
$$\BGGcat_{\w{\lambda}} \isoto \mmOD(\w{{\lambda}}, G,
\Ug^{\lambda'}) \hbox{ and } \BGGcat_{\w{\lambda +
\Lambda}} \isoto \mmOD(G, \Ug^{\lambda'}),$$
both given by $V \mapsto (\BGG_G \ot
V)^B$. Thus $\BGGcat_{\w{\lambda}} \cong {}_{\lambda'} \Hcal_{\w\lambda}$. See \cite{BerGel},  \cite{Soe86}.
\end{example}

\begin{rem}
$\mmOD(\w{{\lambda}}, \uo, P, \D^{\lambda'}_\B)$
will \textit{not} consist of holonomic $\D$-modules, unless $P = B$. For
instance, if $\lambda = - \rho$, $P = G$ and $\lambda' = 0$, then $\BGGcat_{\w{-\rho}}$ will consist of
direct sums of copies of the simple Verma module $M_{-\rho}$.
Corresponding to $M_{-\rho}$ is a non-holonomic submodule of the
$\D_\B$-module $\D_\B$ (see \ref{inj1map}).
\end{rem}

\section{Whittaker modules}\label{Whittaker modules} Let $f : \U(\n) \to \C$ be an algebra homomorphism,  $\Delta_f := \{\alpha \in \Delta;  f(X_\alpha) \neq 0\}$ and $J_f :=
\Ker f$.  Let $\wt{\N}_f := \wt{N}(\g)_f$ be the category of left $\U$-modules on which $J_f$ acts locally nilpotently
and let $\N_f$ be its subcategory of modules which are f.g. over $\U$. Objects of $\N_f$ are called Whittaker modules.
Replacing $\g$ by $\lo$ and $f$ by $f |_{\U(\n \cap \lo)}$ we get the category $\N_f(\lo)$.
For regular $f$, i.e. when $\Delta_f = \Delta$, it was studied by Kostant, \cite{K78};
he showed that $\N_f$ has the exceptionally simple description
\begin{equation}\label{Kostant}
\MOD(\Z) \isoto \N_f, \, M \mapsto M \ot_{\Z} \U/\U  \cdot J_f.
\end{equation}
In the other extreme, when $f = 0$, $\N_f$ is $\BGGcat$ with the $\h$-semi-simplicity condition dropped and it has the same simple objects as $\BGGcat$.

Our main result here is a new proof of Theorem \ref{MainWhittakerTheorem} of \cite{MS97}. It enables one to compute the characters of standard Whittaker modules by means of the Kazhdan-Lusztig conjectures.
(For non-integral weights they were computed in \cite{B97}.)

\smallskip

Throughout this section we assume $\lambda \in \h^*$ and $\Delta_P = \Delta_f = \Delta_\lambda$.
\subsection{Equivalence between a block of $\N_f$ and of singular $\BGGcat$}\label{Whittaker one}
Fix a charcater $f: \U(\n) \to \C$.
For $\mu \in \h^*$ we put
$${}_\mu \N_f := \{M \in \N_f; I_\mu M = 0\}, \ {}_{\w\mu} \N_f := \{M \in \N_f; \; I_\mu \hbox{ acts locally nilpotently on  $M$} \}.$$
(Categories ${}_\mu \wt{\N}_f$ and ${}_{\w\mu} \wt{\N}_f$ are similarly defined.)
Our aim is to prove

\begin{thm}\label{MainWhittakerTheorem} Assume that $\lambda, \lambda' \in \Lambda$ satisfies
$\Delta_f = \Delta_\lambda$ and that $\lambda'$ is regular dominant. Then $\BGGcat_{\w{\lambda}} \cong {}_{\lambda'} \N_f$.
\end{thm}
Before proving this we establish some preliminary results.
\begin{lem}\label{WhittakerAlemma} \textbf{i)} For each $\mu, \lambda \in \h^*$, $\mu$ dominant, such that $\W_\mu \subseteq \W_\lambda$,
${}_{\mu} \Hcal_{\w \lambda}$ identifies with a finite length subcategory
of $\BGG_{\w \lambda}$ which is non-zero iff $\lambda - \mu$ is integral (analogous statements hold with $\mu$ and/or $\lambda$
replaced by $\w\mu$ and/or $\w\lambda$).

 \textbf{ii)}   ${}_{\mu} \Hcal_{\w{-\rho}} \cong \mmOD(\C)$ and ${}_{\mu} \wt{\Hcal}_{\w{-\rho}} \cong \MOD(\C)$, for $\mu$ integral.

 \textbf{iii)}   $ \Hcal_{\Zfin}$ is a finite length category.
\end{lem}
\begin{proof} That ${}_{\mu} \Hcal_{\w\lambda} = 0$ if $\mu-\lambda$ is not integral is a consequence of the fact that any $G$-module is a sum of $G$-modules with
integral central characters.

On the other hand, let $\mu-\lambda$ be integral and $E$ be an irreducible $G$-module with extremal weight $\mu-\lambda$. For $M \in \Hcal_{\lambda}$ we have $E \ot M \in \Hcal_\lambda$, with respect to the diagonal left $\U$-action
and the right $\U$-action on  the second factor. Thus,
$T^\mu_\lambda M = pr_{\w\mu} (E \ot M) \in {}_{\w \mu} \Hcal_\lambda$.  (Similarly, with $\lambda$ replaced by $\w \lambda$.)

Now $\U^{\lambda} \in {}_{\lambda} \Hcal_{{\lambda}}$ with its natural bimodule structure. Since $\W_\mu \subseteq \W_\lambda$ it is known that $T^\mu_\lambda$ is faithful. Hence
we get $0 \neq T^{\mu}_{\lambda}(\U^{\lambda}) \in {}_{\w\mu} \Hcal_{\lambda}$. Thus, also ${}_{\mu} \Hcal_{\lambda}$ and  ${}_{\mu} \Hcal_{\w\lambda}$ are non-zero.
We have
$${}_{\mu} \Hcal_{\w\lambda} \cong \mmOD({\lambda}, G, \U^{\mu}) \overset{\Loc}{\longrightarrow}
\mmOD({\lambda}, G, \D_G, B, \mu)  \cong $$
$$\mmOD(\mu, B, \D_G, G, {\w\lambda}) \cong \mmOD^{\w\lambda} (\mu,B, \U)= \BGGcat_{\w\lambda}.$$
Since $\mu$ is dominant we have $\Gamma \circ \Loc = Id$.
Since  $\BGGcat_{\w \lambda}$ is a finite length category this implies ${}_\mu \Hcal_{\w \lambda}$ is dito as well.  This proves \textbf{i)}. Moreover,
the fact that $\BGG_{\w{-\rho}} \cong \mmOD(\C)$ now implies
${}_{\mu} {\Hcal}_{\w{-\rho}} \cong \mmOD(\C).$  A similar argument shows  ${}_{\mu} \wt{\Hcal}_{\w{-\rho}} \cong \MOD(\C)$. This proves \textbf{ii)}.

By \ref{center both sides},  $\Hcal_{\Zfin} = {}_{\Zfin} \Hcal_\Zfin$. Since  ${}_{\mu} \Hcal_{\lambda}$ is a finite length category for all $\mu, \lambda \in \h^*$ a devissage implies \textbf{iii)}.
\end{proof}

\begin{lem}\label{WhittakerClemma}  Let $\mu \in \Lambda$.
The functors $\Theta_\mu := (  \  )\ot_{\U(\n \cap \lo)} \C_f : {}_\mu \wt{\Hcal}(\lo,\lo)_{\w{\lambda}} \to {}_\mu {\wt{\N}(\lo)_f}$ and
 $\Theta_{\w\mu} :=  (  \  )\ot_{\U(\n \cap \lo)} \C_f : {}_{\w \mu} \wt{\Hcal}(\lo,\lo)_{\w{\lambda}} \to {}_{\w \mu} {\wt{\N}(\lo)_f}$ are equivalences of categories.
\end{lem}
\begin{proof}  This certainly holds for $\lo = \h$ and from that we immediately reduce to the case $\g = \lo$, $\Delta_f = \Delta$ and $\lambda = - \rho$.
We must then show that the functor
$$\Theta_{\mu}: {}_\mu \wt{\Hcal}_{\w{-\rho}} \to {}_\mu {\wt{\N}_f},  \ M \mapsto M \ot_{\U(\n)} \C_f,$$
is an equivalence of categories.
It follows from Kostant's equivalence \ref{Kostant} that  ${}_\mu {\wt{\N}_f}$ is equivalent to $\MOD(\C)$ (for all $\mu \in \h^*$).
By Lemma \ref{WhittakerAlemma} \textbf{ii)}  also  ${}_\mu \wt{\Hcal}_{\w{-\rho}} \cong \MOD(\C)$; hence it suffices to show that $\Theta_\mu$ takes simples to simples.
The $\Theta_\mu$'s commutes with translation functors,
so since $\U^{-\rho} \in {}_{-\rho} \Hcal_{\w{-\rho}}$ we get
$$\Theta_\mu T^\mu_{-\rho}(\U^{-\rho}) = T^{\mu}_{-\rho} \Theta_{-\rho}(\U^{-\rho}) = T^{\mu}_{-\rho}(U^{-\rho} \ot_{\U(\n)} \C_f).$$
By \cite{K78} the latter is simple. This implies both that  $T^\mu_{-\rho}(\U^{-\rho})$ is simple generator for $ {}_\mu \wt{\Hcal}_{\w{-\rho}}$ and that $\Theta_\mu$ takes simples to simples.  Thus $\Theta_\mu$ is an equivalence.

A devissage using Lemma \ref{WhittakerBlemma} now shows that $\Theta_{\w\mu}$ is an equivalence.
\end{proof}

\begin{lem}\label{WhittakerBlemma} Each $M \in  \wt{\Hcal}_{\w{-\rho}}$ which is countably generated as a left $\U$-module  is faithfully flat as a right $\U(\n)$-module.
\end{lem}
\begin{proof}
Assume first that $M$ is simple. Then it follows from Schur's lemma that $M \in  {}_{\mu} \Hcal_{\w {-\rho}}$, for some integral $\mu
\in \h^*$.  By Lemma \ref{WhittakerAlemma} we know that $ {}_{\mu} \Hcal_{\w{-\rho}} \cong \mmOD(\C)$. Hence, $M \cong T^{\mu}_{-\rho}(\U^{-\rho})$ as this is simple (and hence a simple generator for  ${}_{\mu} \Hcal_{\w{-\rho}} $)
by the proof of Lemma \ref{WhittakerClemma}.  By an adjunction argument $M$ is projective
as a right $\U^{-\rho}$-module. By Kostant's separation of variables theorem, \cite{K63}, $\U^{-\rho}$ is free over $\U(\n)$. Hence $M$ is projective over $\U(\n)$.

Assume now that $M \in   \Hcal_{\w{-\rho}}$  is finitely generated. By Lemma \ref{WhittakerAlemma}  $M$ has finite length and an induction
on its length shows that $M$ again is projective as a right $\U(\n)$-module.

For arbitrary $M$ choose a filtration  $M_0 \subseteq M_1 \subseteq M_2 \subseteq \ldots \subseteq M$ of finitely generated submodules. Put
$\overline{M_i} = M_i/M_{i-1}$. Since all $M_i$ and $\overline{M_i}$ are projective we get that $M_i \cong \oplus_{j \leq i} M_j$ and thus
$$M = \underrightarrow{\lim} \, M_i \cong \underrightarrow{\lim} \oplus_{j \leq i}  \overline{M_j} = \oplus_{i \in \mathbb{N}} \overline{M_i}$$
is projective, and therefore flat, as a right $\U(\n)$-module.

To see that
$M$ is faithful over $\U(\n)$, we observe that the above implies that $M$, as a right $\U(\n)$-module, is a direct sum of modules of the form $T^{\mu}_{-\rho}(\U^{-\rho})$,  so
it suffices to show that $T^{\mu}_{-\rho}(\U^{-\rho})$ is faithful over $\U(\n)$. Let $V \in \MOD(\U(\n))$ be non-zero. We have
$$
T^{\mu}_{-\rho}(U^{-\rho}) \ot_{\U(\n)} V  \cong T^{\mu}_{-\rho}(U^{-\rho} \ot_{\U(\n)} V) \neq 0,
$$
since $U^{-\rho} \ot_{\U(\n)} V  \neq 0$ and $T^{\mu}_{-\rho}$ is faithful (since $\W_{\mu} \subseteq \W_{-\rho}$).
\end{proof}

\begin{lem}\label{WhittakerDlemma} Let $\mu \in \Lambda$ and  $M \in {}_{\w\mu} {\N}_f$. Then $M = \oplus_{\nu \in \Lambda} pr_{\lo, \w{\nu}}   M$.
\end{lem}
\begin{proof}  Note that $M$ has a filtration $M_0 \subseteq M_1 \subseteq \ldots \subseteq M_n = M$
 such that each subquotient $\overline{M}_i := M_i/M_{i-1}$ is generated over $\U$ by a vector $v_i$ such that $J_f \cdot v_i = I_\mu \cdot v_i = 0$. Thus each  $\overline{M}_i$ is a quotient of a sum of copies of  $\U^\mu/\U^\mu \cdot J_f$
 and by \cite{MS97} the latter has a filtration with subquotients of the form $\U^{\mu}/\U^\mu(I_{\lo, w \cdot \mu} + J_f)$, $w \in \W$.
 These are in turn quotients of  $\U^{\mu}/\U^\mu \cdot I_{\lo, w \cdot \mu}$.
 Thus, it is enough to prove that
$$
 \U^{\mu}/\U^\mu \cdot I_{\lo, w \cdot \mu}  =   \oplus_{\nu \in \Lambda} pr_{\lo, \w{\nu}}   \U^{\mu}/\U^\mu \cdot I_{\lo, w \cdot \mu} , \ w \in \W.
$$
Since ${}_{\w \nu} \Hcal(\lo,\lo)_{w \cdot \mu} = 0$, for $\nu \notin w \cdot \mu + \Lambda = \Lambda$,  and since  $\U^{\mu}/\U^\mu \cdot I_{\lo, w \cdot \mu} \in  \wt{\Hcal}(\lo,\lo)_{w \cdot \mu}  =  {}_{\Z(\lo)\text{-}\fin}\wt{\Hcal}(\lo,\lo)_{w \cdot \mu}$
we are done.
\end{proof}

\begin{proof}[Proof of Theorem \ref{MainWhittakerTheorem}] We have $\BGGcat_{\w \lambda} \cong  {}_{\lambda'} \Hcal(\lo)_{\w \lambda}$, so we need to construct an equivalence
\begin{equation} \Theta:  {}_{\lambda'} \Hcal(\lo)_{ \w \lambda} \isoto {}_{\lambda'} \N_f, \ M \mapsto M \ot_{\U(\n \cap \lo)} \C_f.
\end{equation}
Consider the restriction functor $res: {}_{\lambda'} \Hcal(\lo)_{\w \lambda} \to \wt{\Hcal}(\lo,\lo)_{\w \lambda}$.
A ``reductive version" of Lemma \ref{WhittakerBlemma} applied to $\lo$ shows that each object of $\Hcal(\lo,\lo)_{\w \lambda}$ is faithfully flat as a right $\U(\n \cap \lo)$-module. Hence, $\Theta$ is faithful and exact.

Denote by $\Psi$ the right adjoint of $\Theta$. Thus
$$\Psi V = \Hom_\C(  \underleftarrow{\lim}_i   \U(\lo)/(I_{\lo,\lambda})^i \ot_{\U(\n \cap \lo)} \C_f,  V)^{\lo\text{-}\integ},$$ where  $( \ )^{\lo\text{-}\integ}$
is the functor that assigns a maximal $\lo$-integrable sub-object. (The left $\U$-module structure on $\Psi V$ comes from the left $\U$-action on $V$ and its right
$\U(\lo)$-module structure comes from the left $\U(\lo)$-action on $ \underleftarrow{\lim}_i   \U(\lo)/(I_{\lo,\lambda})^i \ot_{\U(\n \cap \lo)} \C_f$.)

In order to prove that $\Theta$ is an equivalence its enough to show that the natural transformation
$\Theta \circ \Psi \to Id$ is an isomorphism. Take $V \in {}_{\lambda'}\N_f$ and put  $$K := \Ker\{\Theta\Psi V \to V\}, \ C := \Coker\{\Theta\Psi V \to V \}.$$
By Lemma \ref{WhittakerDlemma} we have  $K = \oplus_{\nu \in \Lambda} pr_{\lo, \w{\nu}}   K$ and  $C = \oplus_{\nu \in \Lambda} pr_{\lo, \w{\nu}}   C$. Let $\Psi_{\w\nu}$ be the right adjoint of the functor
$\Theta_{\w\nu}$ from Lemma \ref{WhittakerClemma}. Note that $pr_{\lo, \w \nu} V \in {}_{\w \nu}\wt{\N}(\lo)_f$ and that
$pr_{\lo, \w\nu} K =  \Ker\{\Theta_{\w\nu} \Psi_{\w\nu}  pr_{\lo, \w\nu} V \to pr_{\lo, \w\nu}  V\}$ and
 $pr_{\lo, \w\nu} C =  \Coker\{\Theta_{\w\nu} \Psi_{\w\nu}  pr_{\lo, \w\nu} V \to pr_{\lo, \w\nu}  V\}$.

Assume $\nu \in  \Lambda$. Then
$\Theta_{\w\nu}$ is an equivalence of categories, by Lemma \ref{WhittakerClemma}, and hence we have $pr_{\lo, \w\nu} K = pr_{\lo, \w\nu} C = 0$. Thus $K = C = 0$,  by  Lemma \ref{WhittakerDlemma},  and consequently $\Theta$ is an equivalence.
 \end{proof}
\subsection{Singular and parabolic case.}\label{Whittaker two}  Let $Q \subseteq G$ be a parabolic,  $\qo := Lie \, Q$, $\Qo := G/Q$ and  $I^\qo := \Ker\{\U \to \D(G/Q)\}$.
It is known that $I^{\qo} = \Ann_{\U} (\U \ot_{\U(\qo)} \C)$, $\U/I^{\qo} \isoto \D(\Qo)$, and there is a parabolic version of (regular) Beilinson-Bernstein localization:
$\MOD(\D_G,Q,\qo) \cong \MOD(\D(\Qo))$, \cite{BorBr82}. Let $\BGGcat^\qo := \{M \in \BGGcat; \qo \hbox{ acts locally finitely on } M\}$ be $\qo$-parabolic category $\BGG$,
${\BGGcat}^{\qo}_\lambda := {\BGGcat}^{\qo} \cap {\BGGcat}_\lambda$ and
${\BGGcat}^{\qo}_{\w \lambda} := {\BGGcat}^{\qo} \cap {\BGGcat}_{\w\lambda}$.

All results from Section \ref{section O} extend to these categories. We assume here for simplicity that $\lambda$ is integral and so we can take  $\lambda' := 0$. Then
\begin{equation}\label{parblock one}
{\BGGcat}^{\qo}_\lambda = \mmOD(\qo, Q,  \Ul), \  {\BGGcat}^{\qo}_{\w\lambda} = \mmOD^{\w\lambda}(\qo,Q, \U).
\end{equation}
Like before we get (with self-explaining notations)  $${\BGGcat}^{\qo}_{\w\lambda} \cong \mmOD(\qo, Q, \D_G, P, \uo_P, \w\lambda) \cong$$
$$\mmOD(\w{\lambda},\uo_\p ,P,\D_G, Q, \qo) \cong \mmOD(\w{\lambda},\uo_\p ,P,\D(\Qo)) \cong \Hcal(\D(\Qo), \lo_P)_{\w\lambda}.
$$
Here $\Hcal(\D(\Qo), \lo_P)_{\w\lambda}$ is the category of $\D(\Qo)\text{-}\U(\lo_P)$-bimodules on which the adjoint
$\lo_P$-action is integrable, $I_{\lo,\lambda}$ acts locally nilpotently from the right and $\uo_P$ acts locally nilpotently from the left. Let $\N^{\qo}_f := \{M \in \N_f;  I^{\qo} M = 0\}$.
Thus the equivalence of Theorem \ref{MainWhittakerTheorem}
induces an equivalence
\begin{cor}\label{parabolic Whittaker cor}  (\cite{W09}.) ${\BGGcat}^{\qo}_{\w\lambda} \cong \N^{\qo}_f$.
\end{cor}

\end{document}